\documentclass[10pt]{article}
\usepackage{epsf,times,t1enc,amsfonts,amsmath,verbatim,amssymb,epsfig,graphics,psfrag,mathdots,color,authblk,fullpage}
\title{A fast alternating projection method for complex frequency estimation.}
\author[*]{Fredrik Andersson}
\author[$\dagger$]{Marcus Carlsson}
\author[**]{Per-Anders Ivert}

\affil[*]{Centre for Mathematical Sciences, Lund University, Box 118, SE-22100, Lund, Sweden \authorcr {\tt fa@maths.lth.se}, phone~+46462220912, fax~+46462224010}
\affil[$\dagger$]{ Departamento de Matema'ticas, Universidad de Santiago de Chile. \authorcr Avenida
Alameda Libertador Bernardo O'Higgins 3363. Estación Central. Santiago. Chile \authorcr {\tt marcus.carlsson@usach.cl}, phone~+56975829858}
\affil[**]{Dag Hammarskjölds väg 5i, SE-224 64 Lund, Sweden  \authorcr {\tt pa.ivert@gmail.com}, phone~+4646131077}

\begin{document}

\maketitle
\thispagestyle{empty}
\let\oldthefootnote\thefootnote
\renewcommand{\thefootnote}{\fnsymbol{footnote}}
\let\thefootnote\oldthefootnote

\def\P{{\mathsf P}}
\def\B{{\mathcal B}}
\def\O{{\mathcal O}}
\def\K{{\mathcal K}}
\def\H{{\mathcal H}}
\def\A{{\mathcal A}}
\def\X{{\mathcal X}}
\def\G{{\mathcal G}}
\def\n{{\mathcal N}}
\def\r{{\mathcal R}}
\def\I{{\mathcal I}}
\def\N{{\mathbb N}}
\def\m{{\mathbb M}}
\def\C{{\mathbb C}}
\def\T{{\mathbb T}}
\def\R{{\mathbb R}}
\def\Z{{\mathbb Z}}
\def\d{{\mathcal D}}
\def\D{{\mathbb D}}
\def\M{{\mathcal M}}
\def\L{{\mathcal L}}
\def\e{{\mathcal E}}
\def\J{{\mathcal J}}
\def\f{{\mathcal F}}
\def\z{{\mathcal Z}}
\def\s{{\mathcal S}}
\def\c{{\mathcal C}}
\def\U{{\mathcal U}}
\def\t{{\mathcal T}}

\newcommand{\Rank}{\mathsf{Rank}~ }
\newcommand{\Ran}{\mathsf{Ran}~ }
\newcommand{\dist}{\mathsf{dist} }

\renewcommand{\H}{\mathcal{H}}
\newcommand{\argmin}[1]{\underset{#1}{\mathrm{argmin}}}
\newtheorem{theorem}{Theorem}
\newtheorem{corollary}{Corollary}

\newtheorem{lemma}{Lemma}
\newtheorem{definition}{Definition}
\newtheorem{remark}{Remark}
\newtheorem{proposition}{Proposition}
\newtheorem{algorithm}{Algorithm}
\newenvironment{proof}{\par\noindent\textbf{Proof:}}
\makeatletter
\def\endexempel{\hfill\qed\@endtheorem}
\makeatother

\begin{abstract}
The problem of approximating a sampled function using sums of a fixed number of complex exponentials is considered. We use alternating projections between fixed rank matrices and Hankel matrices to obtain such an approximation. Convergence, convergence rates and error estimates for this technique are proven, and fast algorithms are developed. We compare the numerical results obtain with the MUSIC and ESPRIT methods.
\end{abstract}

\section{Introduction}
The present paper is devoted to the problem of approximating a sampled function with a sum of a given number of complex exponentials. Our approach is based on the fact that if such a sum is used as a generating function for a Hankel matrix, then that Hankel matrix will (generically) be of rank $k$. Using this fact, we develop a method for the detection of complex frequencies from a signal by alternating projections: we project the corresponding Hankel matrix onto the class of symmetric rank $k$-matrices, project the projection on the class of Hankel matrices, and so on. By a {\em complex frequency}, we refer to the coefficient $\zeta\in\mathbb{C}$ in an exponential of the
form $t\mapsto e^{\zeta t}$. 

There are several alternative techniques for the estimation of (complex) frequencies from a signal. Two of the most commonly used ones are
multiple signal classification (MUSIC) \cite{Schmidt,Bienvenu} and estimation using rotational invariance (ESPRIT) \cite{ESPRIT}. The
MUSIC method is a generalization of the Pisarenko method \cite{Pisarenko}. Recently, complex frequency estimation has been
used in the construction of close to optimal quadratures, for instance for bandlimited functions \cite{Beylkin_Monzon_2005}. This
work is related to the work of Adamjan, Arov and Krein \cite{AAK}, and the algorithms described in \cite{Beylkin_Monzon_2005} have been
investigated in more detail in \cite{JAT}.

The technique of alternating projections is generally described as follows: Given two manifolds $\M_1, \M_2 \subset \mathcal{K}$ (where $\mathcal{K}$ is some Hilbert space) and some point $x_0 \in \mathcal{K}$, find a point $x\in \M_1 \cap \M_2$ that is close to $x_0$, by projecting alternately onto $\M_1$ and $\M_2$, respectively. It was proven by von Neumann \cite{vonNeumann} that if $\M_1$ and $\M_2$ are affine linear subspaces, then the sequence of alternating projections
$$\pi_1(x_0), \pi_2(\pi_1(x_0)), \pi_1(\pi_2(\pi_1(x_0))), \dots $$
converges to an optimal solution $x\in \M_1 \cap \M_2$, i.e., one that minimizes $\|x - x_0\|$.

The extension to the case where $\M_1$ and $\M_2$ are convex sets has been extensively treated for a number of applications, cf. \cite{Bauschke93, Bauschke96} and the references therein. Another generalization was given in \cite{lewis_malick}, where the convergence of the alternating projection scheme was proven for the case where, loosely speaking, the tangent spaces of $\mathcal{M}_{1}$ and $\mathcal{M}_{2}$ together span $\mathcal{K}$. Note that only convergence to some point in $\M_1 \cap\M_2$ can be proven, but that this point is not necessarily the point in $\M_1 \cap \M_2$ that is closest to $x_0$.

Moreover, neither of the cases above apply to the case which we are interested in, as the space of rank $k$-matrices is not convex, and the spanning condition is typically far from satisfied. In \cite{altprojection}, convergence of alternating projections between two manifolds is proved under much milder conditions than the ones given in \cite{lewis_malick}. In this paper we prove that these conditions are generically satisfied in our case; complex symmetric rank $k$-matrices and Hankel matrices. Moreover, through the framework of \cite{altprojection} we can provide estimates for how far away from the initial (sampled) function the approximating $k$-term complex exponential sum will be.

The idea of using alternating projections for frequency estimation has appeared in a number of different settings. The method of alternating projection is commonly referred to as Cadzows method in the signal processing community. In \cite{cadzow}, Zangwill's global convergence theorem is used to prove convergence for algorithms with alternately projects onto (possibly more than two) manifolds. However, Zangwill's theorem only provides the existence of a convergent subsequence, and the results in \cite{cadzow} do not give any information on whether or not the point of convergence is close to the original one, cf. \cite{structured_low_rank}. In the paper \cite{cadzow}, several applications are mentioned; one of them is the projection between finite rank matrices and Toeplitz matrices. Toeplitz matrices appear in the estimation of exponentials by using infinite measurement (or expected value) of autocorrelation matrices. For an (infinitely dense) sampling of a function consisting of $k$ complex exponentials, it is possible to form a Toeplitz matrix from which the $k$ frequencies can be recovered. It is worth mentioning that for a finite sampling of a function with $k$ complex frequencies, the resulting autocorreletion function will not have a Toeplitz structure, and hence the frequencies can not be exactly recovered with this method, even in the absence of noise. A survey of problems of approximations using a combination of structured matrices and low-rank matrices is given in \cite{structured_low_rank_survey}. Alternating projections is mentioned as one of the numerical methods for finding approximate solutions.

The use of alternating projections (Cadzow's method) between Hankel and low-rank matrices has appeared several times in the signal processing literature \cite{Liu_altproj_hankel, Lu_altproj_hankel, Prabhu_altproj_hankel}. The approaches differ in the way the complex frequencies are estimated, once the alternative projection method has converged.

In this paper we develop fast methods for the projection steps. We make use of the fact that multiplication by a Hankel matrix, as well as the projection of low rank matrices onto Hankel matrices, can be computed in a fast manner by the use of FFT. For the projection onto low-rank representations, we will use a customized complex symmetric version of the Lanzcos algorithm.\\

Finally, we consider the approximation by exponentials for a particular class of weighted spaces -- including (approximate) Gaussian weights. Let $w$ be a nonnegative function on $\R$ with support $[-1,1]$ and let
\begin{equation}\label{MT1}\omega=w*w.\end{equation}
Let $L^2(\omega)$ denote the set of functions for which $\|f\|_\omega^2=\int_{-2}^2 |f(t)|^2\omega(t)~dt<\infty.$ Given
$f\in L^2(\omega)$ and $k\in\mathbb{N}$, we are interested in computationally efficient methods for finding the best (or close to
best) approximation of $f$ by functions of the form $\sum_{j=1}^k c_j e^{\zeta_j t}$. In this paper we develop a theory for finite
sequences rather than functions on a continuum. Using techniques similar to those developed in \cite{JAT}, it seems to be possible to develop a similar technique for the approximation of functions on a continuum by a finite number of complex exponentials.

\section{Preliminaries}

In section \ref{s1} we give the necessary tools for projection onto matrices of a certain rank and set up the spaces we will work
with. In section \ref{s2} we describe how to go from a Hankel matrix to its symbol and back, in these spaces.

\subsection{Takagi factorization and the Eckart-Young theorem}\label{s1}
We use the notation $\m_{M,N}$ to denote the Hilbert space of $M\times N$ matrices with complex entries, equipped with the
Frobenius norm, given by
\begin{equation}
\|A\|^{2}=\sum_{j=1}^{M}\sum_{l=1}^{N}|A(j,l)|^2.
\end{equation}

Complex symmetric matrices satisfy the symmetry condition $A=A^T$, which is different from the usual (Hermitian)
self-adjointness condition $A=A^\ast$. Similarly to real symmetric matrices, which are always diagonalizable, complex symmetric
matrices can be decomposed as
\begin{equation*}
A=\sum_{m=1}^{N} s_m \overline{u_m}u_m^{*}, \quad
s_m\in \R^+, \quad u_m \in \C^{N},
\end{equation*}
where the vectors $\{u_m\}_m$ are mutually orthogonal. (As usual, elements of $\C^{N}$ are identified by column matrices, and $u^{*}$ is the adjoint, i.e. the transpose of the complex conjugate of $u$.)
This decomposition of $A$ is called a \emph{Takagi} factorization. Note that in contrast to the
Hermitian case, the numbers $s_m$ are nonnegative. Moreover, the vectors $u_m$
satisfy the relation
\begin{equation}\label{eq:con}
Au_{m}=s_{m}\overline{u_m}.
\end{equation}
In \cite{Horn}, the vectors $u_m$ are referred to as \emph{con-eigenvectors} and the positive numbers $s_m$ are referred
to as \emph{con-eigenvalues}. However, the con-eigenvectors are simply singular vectors (obtained from the Singular Value
Decomposition), and the con-eigenvalues are the singular values. This is seen by noting that
$$s^{2}_{m}u_{m}=A^{*}Au_{m}.$$
The converse is not true, since it is easily seen that e.g. $iu_{m}$
fails to be a con-eigenvector but is still a singular vector. However, in the case where the $s_m$'s are distinct and
$(u_m)_{m=1}^{N}$ is any basis of singular vectors, then one can choose $\theta_m\in [0,2\pi)$, such that
$(e^{i\theta_m}u_m)_{m=1}^{N}$ are con-eigenvectors. For the purposes of this
paper, we are only interested in the zeroes of the corresponding polynomials, and hence the $\theta_m$'s have no importance, but it
will be computationally more convenient to extract the con-eigenvectors, and we have thus chosen to use this terminology.

We recall the Eckart-Young theorem (see e.g. \cite[p 205]{Horn}, \cite{eckart_young}),
(usually stated using the singular vectors):
\begin{theorem}\label{Takagi_Eckart_Young}
Let $A\in \m_{N,N}$ be a complex symmetric matrix with distinct
con-eigenvalues. Given a positive integer $k\le N$, the best
rank $k$ approximation of $A$ (in $\m_{N,N}$) is given by
\begin{equation} \label{rank_n_rep}
\sum_{m=1}^{k} s_m \overline{u_m}u_m^{*},
\end{equation}
where $s_m$ and $u_m$ are the (decreasingly ordered) con-eigenvalues and con-eigenvectors of $A$, respectively.
\end{theorem}

The above theorem can clearly be used to project a given matrix onto the closest rank $k$ matrix (with respect to the Frobenius norm).
We will also make use of approximations in weighted spaces. Given a positive weight $w\in \mathbb{R}^{N}$, we denote by $\m^w_{N,N}$ the Hilbert space of matrices with the weighted Frobenius norm, given by
\begin{equation*}
\|A\|^{2}_{w} = \sum_{j,k=1}^{N}w(j)|A(j,k)|^2 w(k)=\|\mathrm{diag}(\sqrt{w})\,A\,\mathrm{diag}(\sqrt{w})\|^{2}.
\end{equation*}

\begin{theorem} \label{Thm_weight_app}
Let $A\in \m^w_{N,N}$, and let $s_m$ and $q_m$ denote con-eigenvalues and con-eigenvectors of $B=\mathrm{diag}(\sqrt{w}) ~ A ~ \mathrm{diag}(\sqrt{w})$. Then the best rank $k$-approximation of $A$ (in $\m^w_{N,N}$) is given by
$$\sum_{m=1}^{k} s_m \overline{u_m} u_m^{*},$$
where $u_m(l)=q_m(l)/\sqrt{w(l)}$, $1\le l \le N$.
\end{theorem}
\begin{proof}
By definition
$$\left\|A-\sum_{m=1}^{k}s_{m}\overline{u_{m}}u_{m}^{*}\right\|_{w}=\left\|\mathrm{diag}(\sqrt{w})\Big(A-\sum_{m=1}^{k}s_{m}\overline{u_{m}}u_{m}^{*}\Big)\,\mathrm{diag}(\sqrt{w})\right\|$$
$$=\left\|B-\sum_{m=1}^{k} s_m \big(\mathrm{diag}(\sqrt{w})\overline{u_{m}}\big)\big(\mathrm{diag}(\sqrt{w})u_{m}\big)^{*} \right\|$$
which according to Theorem \ref{Takagi_Eckart_Young} is minimized by choosing $u_m = (\mathrm{diag}(\sqrt{w}))^{-1}
q_m$ and by choosing $s_m$ as the con-eigenvalues of $B$.
\end{proof}

There are different ways to compute Takagi factorizations. We indicate one method, the first step of which is the following proposition.
\begin{proposition} \label{Takagi_real_compute}
Let $A$ and $B$ be real symmetric $(N\times N)$-matrices and let
$$W=\left( \begin{array}{rr}A&-B\\-B&-A\end{array}\right).$$
Let $d_{1}\geq d_{2}\geq\ldots\geq d_{2N}$ be the eigenvalues of $W$.
Then $d_{j}+d_{2N+1-j}=0$ for $j=1,2,\ldots,2N$, and an orthonormal basis of eigenvectors can be chosen as
$$\left(\begin{array}{c}X_{1}\\Y_{1}\end{array}\right),\left(\begin{array}{c}X_{2}\\Y_{2}\end{array}\right),\ldots,\left(\begin{array}{c}X_{2n}\\Y_{2n}\end{array}\right),$$
where $X_{j},Y_{j}\in\R^{N},\quad X_{2n+1-j}=-Y_{j}$ and $Y_{2n+1-j}=X_{j}$ for $j=1,2,\ldots,N$.
\end{proposition}
The proof is given as an exercise in \cite{Horn}.

\subsection{Hankel matrices}\label{s2}
A \emph{Hankel} matrix $A$ has constant entries on the anti-diagonals, i.e. it satisfies the relation
$$A(j,l)=A(j',l'), \quad \mbox{if $j+l=j'+l'$}.$$
Every Hankel matrix $A\in\m_{N,N}$ can thus be generated
from some vector $f=(f_j)_{j=2}^{2N}$ by
\begin{equation}\label{defH}
A(j,l)=Hf(j,l)=f(j+l), \quad 1\le j,l\le N.
\end{equation}
An orthonormal basis for the Hankel matrices in $\m_{N,N}$ is
given by
\begin{equation} \label{Hankel_basis}
e_m(j,l) = \left\{
          \begin{array}{ll}
            \frac{1}{\sqrt{N+1-|N-m|}}, & \hbox{if $j+l=m$;} \\
            0, & \hbox{otherwise.}
          \end{array}
        \right.
\end{equation}
for $2\le m \le 2N$, where the normalization factor originates from the number of elements along anti-diagonal $m$. When
considering Hankel matrices in weighted spaces we need to use proper normalization; the basis elements should be normalized with respect
to the induced (matrix) weights along the anti-diagonal. We associate the weights
\begin{equation} \label{omegavikt}
\omega(m)=\underset{1\leq j,l\leq N}{\sum_{j+l=m}} w(j)w(l),
\quad 2\le m\le2N,
\end{equation}
to $w$, and note that this can be written as a discrete convolution
$\omega= \tilde{w}\ast\tilde{w}$, where $\tilde{w}$ denotes the zero
padded version of $w$. A basis for Hankel matrices in the weighted
space $\m^w_{N,N}$ is then given by
\begin{equation*}\label{Hankel_viktad_bas}
e^w_m(j,l) = \left\{
          \begin{array}{ll}
            \frac{1}{\sqrt{\omega(m)}}, & \hbox{if $j+l=m$;} \\
            0, & \hbox{otherwise.}
          \end{array}
        \right.
\end{equation*}
for $2\le m \le 2N$. Note that in the case $w=1$ we get
the ``triangle weight'' which appeared in
(\ref{Hankel_basis}). We let $\ell_{2N-1}^\omega$ be the space of complex sequences $f=(f_j)_{j=2}^{2N}$, equipped with the norm defined by
$$\|f\|_{\omega}^{2}=\sum_{j}|f_j|^2\omega(j).$$

The mapping $H$ (given by (\ref{defH})) will in the sequel be considered as a mapping from $\ell_{2N-1}^\omega$ to
$\langle\m_{N,N},\|\cdot\|_{w}\rangle$. It is a unitary map (isometric isomorphism), whose adjoint is
the weighted averaging operator
\begin{equation}\label{8}
H^\ast A(m) = \frac{1}{\omega(m)} \sum_{j+l=m} w(j) A(j,l) w(l),
\end{equation}
and $H^\ast H = I$. The following proposition is now immediate.

\begin{proposition}\label{p1}
Let $w\in\R_+^{N}$ be given and let $\omega$ be the associated
weight defined by (\ref{omegavikt}). Let $f=(f_j)_{j=2}^{2N}$ and let $\mathcal{S}$ be any set
of Hankel matrices. Then the problem
\begin{equation*}
\argmin{\tilde{H} \in \mathcal{S}} \| Hf -\tilde{H}\|_w
\end{equation*}
is equivalent to the problem
\begin{equation*}
\argmin{g \in H^\ast(\mathcal{S})} \| f -g\|_{\ell_{2N-1}^\omega}.
\end{equation*}
The solutions are related by $Hg = \tilde{H}$.
\end{proposition}

\section{Properties of fixed-rank and Hankel matrices} \label{Hankel_properties}
The key observation behind the algorithms of this paper is that a rank $k$ Hankel operator generically has a symbol which is a sum of $k$
exponentials. However, this is not always true, and neither is the projection onto rank $k$ matrices, given by Theorem \ref{Takagi_Eckart_Young}, well defined at all points. In this section we show that the exceptional set is very small. We introduce the concept
of a \textit{thin} set, and show that the exceptional points are confined to thin sets.

We denote by $\H_{N}$ the set of Hankel matrices in $\m_{N,N}$, and $\r_{N,k}$ will denote the
set of matrices in $\m_{N,N}$ of rank at most $k$.

\subsection{Manifold structure}\label{I}
In this entire section, we will work with subsets of $\m_{N,N}$, consisting of matrices
whose entries are ordered from $1$ to $N$. $\H$ is a linear subspace of $\m_{N,N}$
and, hence, a differentiable manifold of (real) dimension $2(2N-1)$.
By identifying $\C$ with $\R^2$ in the obvious way, a simple
modification of $H$ (defined in (\ref{defH})) provides a natural
chart. The structure of $\r_{N,k}$ is more complicated; we will show that it is a
manifold of (real) dimension $2(2Nk-k^2)$ outside a small exceptional set.
Suppose $A\in\r_{N,k}$, and use the singular value decomposition of $A$ to find $\sigma_A\in(\R^+)^k$ and $U_A,V_A$
such that $U_A^*U_A=V_A^*V_A=I_{k}$ (where $I_{k}$ is the $k\times k$ identity matrix and $U_A,~V_A$ are $N\times k$-matrices) and
\begin{equation}\label{parark}A=V_A\left(
                                   \begin{array}{cccc}
                                     \sigma_{A,1} & 0 & \cdots & 0 \\
                                     0 & \sigma_{A,2} & \ddots &\vdots \\
                                     \vdots & \ddots & \ddots & 0 \\
                                     0 & \cdots & 0 & \sigma_{A,k} \\
                                   \end{array}
                                 \right) U_A^*=V_A I_\sigma U_A^*.
\end{equation}
A typical matrix in $\r_{N,k}$ satisfies
\begin{equation}\label{distinct}\sigma_{A,1}>\sigma_{A,2}>\ldots>\sigma_{A,k}>0,\end{equation}
and, if this is not the case, an arbitrary small numerical
perturbation will yield distinct singular values. The subset of
$\r_{N,k}$, consisting of $N\times N$-matrices satisfying (\ref{distinct}), will be denoted $\r_{N,k}^d$, where
$d$ stands for ``distinct''. If $\M$ is a manifold and $E$ is a set, contained in the union of finitely many
manifolds of dimension lower than the dimension of $\M$, we will say that $E$ is \textit{thin relatively to} $M$.
\begin{proposition}
$\r_{N,k}^d\subset\m_{N,N}$ is a manifold of (real) dimension
$2(2Nk-k^2)$. Moreover, $\r_{N,k}=\overline{\r_{N,k}^d}$ and
$\r_{N,k}\setminus\r_{N,k}^d$ is thin relatively to $\r_{N,k}$.
\end{proposition}

\begin{proof}
We start by remarking that the set $\U(N,k)$ of complex $N\times k$-matrices $U$ satisfying $U^{*}U=I_{k}$ is a real manifold of
dimension $2Nk-k^{2}$. Namely, the columns $U(\cdot,1),U(\cdot,2),\ldots,U(\cdot,k)$ of such a matrix can be identified with points on $S^{2N-1}$, and thus $\U(N,k)$ can be identified with the subset of elements $U\in(S^{2N-1})^{k}$, satisfying the functionally independent equations
$$\mbox{\rm Re }U^{*}(\cdot,j)U(\cdot,l)=\mbox{\rm Im }U^{*}(\cdot,j)U(\cdot,l)=0,\quad 1\leq l\leq j\leq k.$$
The number of these equations is
$$2+4+\ldots+2(k-1)=k^{2}-k,$$
and thus $\U(N,k)$ is a manifold of dimension
$$k(2N-1)-(k^{2}-k)=2Nk-k^{2}.$$

Now let $A\in\r_{N,k}^d$. Then there are matrices $U_{A}$ and $V_{A}$ in $\U(N,k)$ and a vector $\sigma_{A}\in\R^{k}_{+}$ with $\sigma_{A,1}>\sigma_{A,2}>\ldots>\sigma_{A,k}$, such that
$$A=V_{A}\,\mathrm{diag}(\sigma_{A})\,U^{*}_{A}=\sum^{k}_{j=1}\sigma_{A,j}V_{A}(\cdot,j)U^{*}_{A}(\cdot,j).$$
In this representation, the numbers $\sigma_{A,j}$ are uniquely determined by $A$, and so are the products $V_{A}(\cdot,j)U^{*}_{A}(\dot,j)$, but the vectors $U_{A}(\cdot,j)$ and $V_{A}(\cdot,j)$ are not; each vector $U_{A}(\cdot,j)$ can be multiplied by a complex unit factor $e^{i\theta_{j}}\in S^{1}$ and $V_{A}(\cdot,j)$ by the same factor, whence the product $V_{A}(\cdot,j)U^{*}_{A}(\cdot,j)$ remains unaffected.
We can thus define a mapping
$$F:\r^{d}_{k}\times(S^{1})^{k}\to\U(N,k)\times\{\sigma\in\R^{k};\,\sigma_{1}>\sigma_{2}>\ldots>\sigma_{k}>0\}\times\U(N,k)$$
by
$$F(A,(e^{i\theta_{j}})^{k}_{j=1})=\Big(V_{A}\mathrm{diag}(e^{i\theta_{j}})^{k}_{j=1},\,\mathrm{diag}(\sigma_{A}),\,U_{A}\mathrm{diag}(e^{i\theta_{j}})^{k}_{j=1}\Big).$$
It is easily verified that this mapping is a diffeomorphism, and hence
$$\dim\r^{d}_{k}+k=(2Nk-k^{2})+k+(2Nk-k^{2}),\quad\mbox{i.e.}\quad\dim\r^{d}_{k}=2(2Nk-k^{2})$$
We omit a proof of the remaining statements, which can be obtained
by standard matrix theory and differential geometry.
\end{proof}

Given a matrix $A\in\m_{N,N}$, the closest point in $\r_{N,k}$ is given
by the Eckart-Young theorem, and it is unique whenever the singular
values are distinct. By the above theorem, it is very improbable
that this would not be the case for an arbitrary matrix $A$. Indeed,
when working with ``real numerical'' data this never happens, so we
will for simplicity treat the projection onto $\r_{N,k}$ as a well
defined map which we denote by $\pi_{\r_{N,k}}$. A more stringent approach would be to work with ``point to set''-maps, as in \cite{cadzow} and
\cite{Zangwill}.

\subsubsection*{Infinite Hankel matrices of finite rank}
To understand the structure of Hankel matrices, it seems indispensable to consider infinte Hankel matrices, by which we mean complex-valued functions on $\mathbf{N}\times\mathbf{N}$, where $\mathbf{N}=\{0,1,2,\ldots\}$ (in this section we include $0$ in the index set for convenience). For a complex valued funktion $f$ on $\mathbf{N}$, we denote by $Hf$ the infinite Hankel matrix with $Hf(j,l)=f(j+l)$. This means that $H$ is an operator from $\C^{\N}$ to $\C^{\N^{2}}$.

The rank of an infinite matrix is the dimension of its column space (the linear space generated by its columns).

Assume that $A=Hf$ is an infinite Hankel matrix, such that some column is a (complex) linear combination of the preceding ones (i.e. $\mathrm{rank}\,A<\infty$). Låt $A(\cdot,r)$ be the fist one of these. It thus holds
$$A(j,r)+\sum^{r-1}_{l=0}\lambda_{l}A(j,l)=0$$
for all $j\in\mathbf{N}$ (where $\lambda_{0},\ldots,\lambda_{r-1}$ are complex numbers), which means that
$$f(j+r)+\sum^{r-1}_{l=0}\lambda_{j}f(j+l)=0,$$
i.e.
\begin{equation}\label{eq:rekursion}
f(k)+\sum^{r-1}_{l=0}\lambda_{l}f(k-r+l)=0,\quad k\geq r.
\end{equation}
Vi find that every column, starting with $A(\cdot,r)$, is a linear combination (with the same coefficients) of the $r$ preceding columns, and we conclude that $r$ is the rank of the matrix.

\begin{theorem}\label{corner}
Let $A$ be an infinite Hankel matrix of rank $r<\infty$. Then
$$\left|\begin{array}{cccc}A(0,0)&A(0,1)&\ldots&A(0,r-1)\\A(1,0)&A(1,1)&\ldots&A(1,r-1)\\\vdots&\vdots&\ddots&\vdots\\A(r-1,0)&A(r-1,1)&\ldots&A(r-1,r-1)\end{array}\right|\neq 0.$$
\end{theorem}

\begin{proof} Assume that the determinant vanishes. We have seen that every column, starting with $A(\cdot,r)$ is a linear combination (with the same coefficients) of the $r$ preceding ones, and in the same way we see that the corresponding relation holds for the rows. It now follows that
$$\left|\begin{array}{cccc}A(j_{1},0)&A(j_{1},1)&\ldots&A(j_{1},r-1)\\A(j_{2},0)&A(j_{2},1)&\ldots&A(j_{2},r-1)\\\vdots&\vdots&\ddots&\vdots\\A(j_{r},0)&A(j_{r},1)&\ldots&A(j_{r},r-1)\end{array}\right|=0$$
whenever $0\leq j_{1}<j_{2}<\ldots<j_{r}$, since every row in this determinant is a linear kombination of the linearly dependent rows $A(0,\cdot),A(1,\cdot),\ldots,A(r-1,\cdot)$. This means that the first $r$ columns of $A$ are linearly dependent, contrary to the observations made above.
\end{proof}

We now study the generating function for $f$:
$$F(x)=\sum^{\infty}_{k=0}f(k)x^{k}$$
Using (\ref{eq:rekursion}), we get
\begin{equation}\label{eq:kvot}
F(x)=\frac{f(0)+\sum^{r-1}_{k=1}\Big(f(k)+\sum^{r-1}_{l=r-k}\lambda_{l}f(k-r+l)\Big)x^{k}}{1+\sum^{r}_{k=1}\lambda_{r-k}x^{k}}.
\end{equation}
In this quotient, the degree of the numerator is at most $r-1$.
If $\lambda_{0}\neq 0$, the degree of the denominator is $r$, and there is an expansion
\begin{equation}\label{eq:partial}
F(x)=\sum^{p}_{\nu=1}\sum^{m_{\nu}-1}_{\mu=0}\frac{a_{\nu,\mu}}{(1-\zeta_{\nu}x)^{\mu+1}},
\end{equation}
where $m_{1}+m_{2}+\ldots+m_{p}=r$, and $a_{\nu,\mu}$ are constants with $a_{\nu,m_{\nu}}\neq 0$. Hence
$$F(x)=\sum^{p}_{\nu=1}\sum^{m_{\nu}-1}_{\mu=0}\frac{a_{\nu,\mu}}{\mu!}\frac{d^{\mu}}{dx^{\mu}}\frac{\zeta^{-\mu}_{\nu}}{1-\zeta_{\nu}x}=
\sum^{p}_{\nu=1}\sum^{m_{\nu}-1}_{\mu=0}\frac{a_{\nu,\mu}}{\mu!}\frac{d^{\mu}}{dx^{\mu}}\sum^{\infty}_{k=0}\zeta^{k-\mu}_{\nu}x^{k}$$
$$=\sum^{p}_{\nu=1}\sum^{m_{\nu}-1}_{\mu=0}A_{\nu,\mu}\sum^{\infty}_{k=\mu}\left(\begin{array}{c}k\\\mu\end{array}\right)(\zeta_{\nu}x)^{k-\mu}=\sum^{p}_{\nu=1}\sum^{m_{\nu}-1}_{\mu=0}a_{\nu,\mu}\sum^{\infty}_{k=0}\left(\begin{array}{c}k+\mu\\\mu\end{array}\right)(\zeta_{\nu}x)^{k}$$
$$=\sum^{\infty}_{k=0}\sum^{p}_{\nu=1}\sum^{m_{\nu}-1}_{\mu=0}a_{\nu,\mu}\left(\begin{array}{c}k+\mu\\\mu\end{array}\right)\zeta_{\nu}^{k}x^{k}.$$
We find that
$$f(k)=\sum^{p}_{\nu=1}Q_{\nu}(k)\zeta_{\nu}^{k},$$
where the $Q_{\nu}$ are polynomials of degree $m_{\nu}-1,\;\nu=1,2,\ldots,p$.

If $\lambda_{0}=0$, the numerator in (\ref{eq:kvot}) is of degree $r-1$, because otherwise the columns $A(\cdot,r-1)$ would be a linear combination of the preceding ones, contrary to our choice of $r$. In this case we let $d$ be the first number with $\lambda_{d}\neq 0$, and a polynomial division yields
$$F(x)=Q(x)+\sum^{p}_{\nu=1}\sum^{m_{\nu}-1}_{\mu=0}\frac{a_{\nu,\mu}}{(1-\zeta_{\nu}x)^{\mu+1}},\;\;\deg Q(x)=d-1.$$
where $m_{1}+m_{2}+\ldots+m_{p}=r-d$, and the $a_{\nu,\mu}$ are constants with $a_{\nu,m_{\nu}}\neq 0$.

If we define $\delta(j)$ as $0$ when $j\neq 0$ and $1$ when $j=0$, we can write

\begin{theorem}
Let $A=Hf$ be an infinite Hankel matrix of finite rank $r$. Then
$$f(k)=\sum^{d-1}_{\mu=0}c_{\mu}\delta(k-\mu)+\sum^{p}_{\nu=1}Q_{\nu}(k)\zeta^{k}_{\nu},$$
where $c_{d-1}\neq 0$ (in case $d\geq 1$), $Q_{\nu}$ are polynomials with $\deg Q_{\nu}=m_{\nu}-1$, and $d+m_{1}+m_{2}+\ldots+m_{p}=r$.
\end{theorem}

We can also write
$$f(j+l)=\sum^{d-1}_{\mu=0}¨c_{\mu}\delta(j+l-\mu)+\sum^{p}_{\nu=1}\sum^{m_{\nu}-1}_{\mu=0}q_{\nu,\mu}(j)\zeta^{j}_{\nu}l^{\mu}\zeta^{l},\quad\deg q_{\nu.\mu}=m_{\nu}-1-\mu.$$

We will now investigate how the nodes $\zeta_{\nu}$ can be determined. We put
$$P(x)=x^{d}\prod^{p}_{\nu=1}(x-\zeta_{\nu})^{m_{\nu}}=\sum^{r}_{l=0}\lambda_{l}x^{l}.$$
Then, for $\nu=1,2,\ldots,p$,
$$(x\frac{d}{dx})^{\mu}P(x)_{|x=\zeta_{\nu}}=0,\quad\mu=0,1,\ldots,m_{\nu}-1,$$
i.e.
$$\sum^{r}_{l=0}\lambda_{l}l^{\mu}\zeta^{l}_{\nu}=0,\quad\mu=0,1,\ldots,m_{\nu}-1.$$
It also holds $\lambda_{l}=0$ if $l<d$. Hence, for $l=0,1,2,\ldots,r$,
$$\sum^{r}_{l=0}f(j+l)\lambda_{l}=\sum^{d-1}_{\mu=0}c_{\mu}\sum^{r}_{l=0}\delta(j+l-\mu)\lambda_{l}+\sum^{p}_{\nu=1}\sum^{m_{\nu}-1}_{\mu=0}q_{\nu,\mu}(j)\zeta^{j}_{\nu}\sum^{r}_{l=0}\lambda_{l}l^{\mu}\zeta^{l}=0$$

Now define, for any nonnegative integer $k$, the {\em upper left corner submatrix} of order $k+1$ by
$$A_{k}=\left(\begin{array}{cccc}A(0,0)&A(0,1)&\ldots&A(0,k)\\A(1,0)&A(1,1)&\ldots&A(1,k)\\\vdots&\vdots&\ddots&\vdots\\A(k,0)&A(k,1)&\ldots&A(k,k)\end{array}\right).$$
We know that $\det A_{r-1}\neq 0$ and $\det A_{r}=0$. Hence the kernel for $A_{r}$ is one-dimensional, and we have characterized it: It is generated by the vector $(\lambda_{0},\lambda_{1},\ldots,\lambda_{r})$, where
$$\sum^{r}_{l=0}\lambda_{l}x^{l}=x^{d}\prod^{p}_{\nu=1}(x-\zeta_{\nu})^{m_{\nu}}.$$
We now observe that the numbers $\lambda_{0},\lambda_{1},\ldots,\lambda_{r}$ are exactly the numbers appearing in (\ref{eq:rekursion}) (with $\lambda_{r}=1$), and using that recursion equation, it is easily seen that for $k\geq r$, the $(k+1-r)$-dimensional kernel of $A_{k}$ is generated by the vectors $(0,\ldots,0,\lambda_{0},\lambda_{1},\ldots,\lambda_{r-1},\lambda_{r},0,\ldots,0)$. The coordinates of these vectors are the coefficients in the polynomials $x^{l}P(x),~l=0,1,\ldots,k-r$. We now summarize the observations made:

\begin{proposition}\label{centralpolynom}
Let $A=Hf$ be an infinite Hankel matrix of rank $r<\infty$.
Then
\begin{itemize}\item[]
\item[] $$f(k)=\sum^{d-1}_{\mu=0}c_{\mu}\delta(k-\mu)+\sum^{p}_{\nu=1}Q_{\nu}(k)\zeta^{k}_{\nu},$$
where $d\geq 0,~c_{d-1}\neq 0$ (in case $d\geq 1$), $Q_{\nu}$ are polynomials with $\deg Q_{\nu}=m_{\nu}-1$ and $d+m_{1}+m_{2}+\ldots+m_{p}=r$.
\item[] If $k\geq r$, the vector $(\mu_{0},\mu_{1},\ldots,\mu_{k})$ belongs to the kernel of $A_{k}$ if and only if there is a polynomial $Q$ of degree at most $k-r$, such that
$$\sum^{k}_{l=0}\mu_{l}x^{l}=Q(x)x^{d}\prod^{p}_{\nu=1}(x-\zeta_{\nu})^{m_{\nu}}.$$
\end{itemize}
\end{proposition}
We call the polynomial
$$P(x)=x^{d}\prod^{p}_{\nu=1}(x-\zeta_{\nu})^{m_{\nu}}$$
{\em the central polynomial} for $A$.

\subsubsection*{Finite Hankel matrices}
For an infinite Hankel matrix of finite rank $r$, we have seen that the upper left corner matrix of order $r$ is non-singular. For finite Hankel matrices, this will not always be the case. Let $A$ be a Hankel matrix of size $N\times N$, i.e a complex.valued function on $\{(j,l)\in\mathbf{N}^{2};\,0\leq j\leq N-1,\;0\leq l\leq N-1\}$, such that $A(j,l)=A(j',l')$ whenever $j+l=j'+l'$. Then there are infinitely many fucnctions $f$ on $\N$, such that $A_{j,l}=f(j+l)$. Such a function $f$ is determined by $A=Hf$ only on the set $\{0,1,\ldots,2N-2\}$. Vi will now discuss ``canonical'' extensions of $A$ to $\mathbf{N}^{2}$.
\begin{theorem} \label{th:leftuppercorner_hankel}
Låt $A=Hf$ be a $N\times N$ Hankel matrix of rank $r<N$ and assume that its upper left corner submatrix $A_{r}$ of order $r$ is non-singular. Then there are uniquely determined constants $\lambda_{0},\lambda_{1},\ldots,\lambda_{r-1}$, such that
\begin{equation}\label{eq:rek2}
f(k)+\sum^{r}_{l=0}\lambda_{l}f(k-r+l)=0,\;\;k=r,r+1,\ldots,2N-2.
\end{equation}
\end{theorem}
\begin{proof}
We have
$$f(k)=A(0,k),\;0\leq k\leq r-1$$
and, since the column $A(\cdot,r)$ is a linear combination of the linearly independent columns $A(\cdot,l)$, $l=0,1,\ldots,r-1$, there are uniquely determined constants $\lambda_{0},\lambda_{1},\ldots,\lambda_{r-1}$, such that
$$f(j+r)=A(j,r)=-\sum^{r-1}_{l=0}\lambda_{l}A(j,l)=-\sum^{r-1}_{l=0}\lambda_{l}f(j+l),\quad 0\leq j\leq N-1.$$
The relation (\ref{eq:rek2}) is thus valid for $k\leq N-1+r$.
Consequently, for $k=r+1,r+2,\ldots,N-1$,
$$A(j,k)=f(j+k)=-\sum^{r-1}_{l=0}\lambda_{l}f(j+k-r+l)=-\sum^{r-1}_{l=0}\lambda_{l}A(j,k-r+l),\quad j=0,1,\ldots,r-1,$$
Followingly the same relation holds for $j=r,\ldots,N-1$, and thus the recursion formula (\ref{eq:rek2}) holds for $k=r,r+1,\ldots,2N-2$.
\end{proof}

A function $f$, satisfying (\ref{eq:rek2}), has of course a unique extension to a function on $\N$, satisfying the same relation. We conclude that if the condition on the upper left corner submatrix is fulfilled, then $A$ has a canonical rank-preserving extension to an infinite Hankel matrix. If not, any extension to an infinite Hankel matrix is necessarily of a strictly higher rank. The first case is of course generic, and the latter case is exceptional. We will limit our attention to the generic case.

\begin{definition} A matrix $A\in\H_{N,r}:=\r_{N,r}\cap\H_{N}$ belongs to the class $\H^{n}_{N,r}$ if
\begin{enumerate}
\item The upper left corner submatrix of order $r$ is non-singular,
\item In the central polynomial $P(x)=x^{d}\prod^{p}_{\nu=1}(x-\zeta_{\nu})^{m_{\nu}}$, we have $d=0$ and $m_{\nu}=1$ for all $\nu$ (and, consequently, $p=r$).
\end{enumerate}
\end{definition}

\begin{theorem}\label{tKronecker}
\begin{itemize}\item[]
\item[] $\H_{N}$ is a real $2(2N-1)$-dimensional linear subspace of $\m_{N,N}$.
\item[] $\H^{n}_{N,r}$ is a real differentiable manifold of dimension $4r$ which is dense in
$\H_{N,r}$. Its complement $\H_{N,k}\setminus\H^{n}_{N,r}$ is thin relatively to $\H^{n}_{N,r}$.
\item[] The map $\pi_{\r_{N,r}}$ is well
defined at all points of $\H^{n}_{N,r}$.
\end{itemize}
\end{theorem}

\begin{proof}
The first statement is obvious. For the second, it is easily seen that the complex numbers $f(0),f(1),\ldots,f(r-1),\lambda_{0},\lambda_{1},\ldots,\lambda_{r-1}$ in (\ref{eq:rek2}) serve as complex coordinates on $\H^{n}_{N,r}$,
and that the exceptional points (corresponding to matrices not in $\H^{n}_{N,r}$) are given by restrictions, confining them to a thin set.
The third statement is immediate by the Eckart-Young theorem.
\end{proof}

\subsection{Extracting frequencies from low rank Hankel matrices}

We note that the second statement in Theorem \ref{tKronecker} can be
seen as a finite-dimensional version of Kronecker's theorem. We will exploit it in order to approximate functions by sums of
$k$ exponentials;
\begin{equation}\label{eq:Kronecker}
f(l)=\sum_{p=1}^{k} c_p e^{\zeta_p l}, \quad c_p,\zeta_p\in \C.
\end{equation}
We choose some positive weight $w$ that gives rise to a weight
$\omega$ through (\ref{omegavikt}). The problem of approximating $f$
by a sum of $k$ exponentials in $\ell^\omega$ is then according to
Proposition~\ref{Thm_weight_app} equivalent to finding the matrix $H f_{opt} \in
\r_{N,k}\cap \H_{N}$ that minimizes $\|Hg-Hf\|_w$.

Let us turn our focus to how to find $c_p$ and $\zeta_p$ in
(\ref{eq:Kronecker}) given $Hf\in\H^{n}_{N,k}$. If $u=(u_{0},u_{1},\ldots,u_{k})$ is a vector in $\C^{k+1}$, we define the polynomial $P_{u}$, \textit{generated by} $u$, by
$$P_{u}(x)=\sum^{k}_{j=0}u_{j}x^{j}.$$
From Proposition~\ref{centralpolynom} it follows that the nodes $e^{\zeta_{p}}$ in (\ref{eq:Kronecker}) are precisely the zeroes of the central polynomial $P(x)$ for $Hf$, and this polynomial is the last common divisor of all the polynomials generated by vectors in the nullspace of $Hf$.
Alternatively, it is the polynomial generated by a single vector, generating the nullspace of $(Hf)_{k+1}$.
This approach is relatively fast (time $\mathcal{O}(k^3)$), but it does not have good numerical stability.
The reason for this is that we use only local data, i.e only $k+1$
elements from each con-eigenvector $u_m$.

A better method is to observe that if $f$ has the form (\ref{eq:Kronecker}), then, due to (\ref{eq:con}),
the con-eigenvectors $u_{1},u_{2},\ldots,u_{k}$ of $Hf$ span the same subspace of $\C^{N}$ as the vectors
$$Z_{l}=(1,e^{\zeta_{l}},e^{2\zeta_{l}},\ldots,e^{(N-1)\zeta_{l}}),\quad l=1,2,\ldots,k.$$
Let $U=(u_1\quad\ldots\quad u_k)\in\m_{N,k}$. We then have $U=ZG$,
where $Z=(Z(\cdot,1),Z(\cdot,2),\ldots,Z(\cdot,k)\in\m_{N,k}$ is the Vandermonde matrix generated by
$e^{\zeta_p}$ ($Z(l,p)=e^{\zeta_p l}$) and $G$ is an invertible matrix in $\m_{k,k}$. For any matrix $A$, we denote by $A_{(j)}$ the matrix that appears when the $j$-th row of $A$ is removed. Clearly, we have that $U_{(1)}=Z_{(1)}G$ and
$U_{(N)}=Z_{(N)}G$. We also note that
$$Z_{(1)}=Z_{(N)}\mathrm{diag}(e^{\zeta_1},\dots,e^{\zeta_k}).$$
Recall that $U_{(N)}$ has a natural left inverse given by
$U_{(N)}^\dagger=(U_{(N)}^\ast U_{(N)})^{-1} U^\ast.$ From the relations above, it
follows that
$$U_{(N)}^\dagger U_{(1)}=(U_{(N)}^\ast U_{(N)})^{-1}U_{(N)}^\ast Z_{(N)}\mathrm{diag}(e^{\zeta_1},\dots,e^{\zeta_k})G.$$
Now
$$(U_{(N)}^\ast U_{(N)})^{-1}U_{(N)}^\ast Z_{(N)}G=(U_{(N)}^\ast U_{(N)})^{-1}U_{(N)}^\ast U_{(N)}=I_{k},$$
and thus
$$U_{(N)}^\dagger U_{(1)}=G^{-1}\mathrm{diag}(e^{\zeta_1},\dots,e^{\zeta_k})G.$$
Hence we can compute the nodes $\zeta_p$ by computing the
eigenvalues of $U_{(N)}^\dagger U_{(1)}$. This method is numerically
stable and can be computed in $\mathcal{O}(N k^2+k^3)$ time.

Once the nodes $\zeta_m$ are found, the problem of find $c_m$
becomes linear, and again it will be sufficient to consider $k$ consecutive elements solve the corresponding linear system.

\section{Alternating projections}\label{altproj}
Given $f\in \ell_{2N-1}^\omega$, the problem of finding the best
approximation in $\ell_{2N-1}^\omega$ of the form
$f_{opt}(l)=\sum_{j=1}^k c_j e^{\zeta_j l}$ is hard. Instead, our
aim is to find an $\tilde{f}_{opt}$ that is close to optimal. We will
do this by employing \emph{alternating projections}. By Proposition
\ref{p1} we know that this problem is equivalent to
\begin{equation}\label{MT3} \argmin{Hg \in \r_{N,k}\cap\H_{N}} \| Hf -
Hg\|_w.
\end{equation}

By starting with $Hf_0=H f$ and alternatively projecting onto the subsets $\r_{N,k}$ and $\H_{N}$, the idea is that the so arising sequence $H
f_m$ will converge to an intersection point $Hf_\infty\in\H_{N,k}=\r_{N,k}\cap\H_{N}$, and moreover that $Hf_\infty$ is in
fact close to the optimal one, $Hf_{opt}$. This idea was investigated in a general framework in \cite{altprojection}. The
main result of \cite{altprojection} roughly says that the above scheme indeed works if we start not too far away from $\H_N$ and
avoid the thin set of bad points related to $\r_{N,k}$ and $\H_N$, (which in practice does not seem to be an issue). As an example we
studied the case of projections between rank $k$ matrices and Hankel matrices in non-weighted spaces. In this paper we make a more thorough
study of this particular application and extend it to weighted spaces. Moreover, we discuss how to use the weighted spaces
for approximating functions by sums of Gaussians,we discuss how to construct fast implementations of this idea, and
finally we will also prove that the framework of
\cite{altprojection} indeed applies.

We now state the main result of \cite{altprojection} in the current framework. Let $P_{\r_{N,k}}$, $P_{\H}$ and $P_{\H_k}$ denote the maps
taking a given matrix $B$ onto the closest point in the respective manifolds. Already here we hit some technical issues. We clearly
have a formula for $P_{\H_{N}}$ since $\H_{N}$ is linear, so $P_{\H_{N}}$ is an orthogonal projection and an explicit formula is given by (\ref{8}).
Concerning $P_{\r_{N,k}}$ we do have a formula for computing it, but the drawback is that if $B$ has singular values of higher
multiplicity, then the map is not well defined. This is a common feature in algorithmic frameworks, and can be dealt with by
introducing point-to-set maps, following \cite{Zangwill}. However, this seems over-ambitious in the current framework, since matrices with
singular values of multiplicity $> 1$ constitute a thin set
(Theorem \ref{tKronecker}), and arbitrary small (numerical) perturbation yields distinct singular values.
Moreover, in \cite{altprojection} we prove that $P_{\H_N}$ is well
defined near ``regular non-tangential'' points of ${\H_N}$, and we
will prove in Appendix \ref{appendix} that the complement of such
points is thin as well. With this in mind, we will from now on treat
$P_{\r_{N,k}}$ and $P_{\H_N}$ as well defined maps. Note that there is no simple way of computing $P_{\H_N}$. We will prove in Appendix
\ref{appendix} that the theory developed in \cite{altprojection} applies in
the present setting. Combined with this, the Theorem 6.1 of
\cite{altprojection} reads;
\begin{theorem}\label{main}
For all $A\in \H_N$ outdside a thin subset, the following is true.
Given any $\epsilon>0$, there exists an $s>0$ such that, for all
$Hf$ with $\|Hf-A\|\leq s$, the sequence of alternating projections
given by $B_0=Hf$ and
\begin{equation}\label{eq44}B_{j+1}=\left\{\begin{array}{cc}
                                    P_{\r_{N,k}}(B_j) & j \text{ is even} \\
                                    P_{\H}(B_j) & j \text{ is odd}
                                  \end{array}\right.
\end{equation}
\begin{itemize}
\item[($i$)] converges to a point $Hf_\infty\in\H^{n}_{N,k}$
\item[($ii$)] $\|Hf_\infty-Hf_{opt}\|\leq\epsilon\|Hf-Hf_{opt}\|$
\end{itemize}
\end{theorem}
A few remarks: $(i)$ combined with Theorem \ref{tKronecker} says that we will achieve an approximation of $f$ of the form
$f_\infty(l)=\sum_{j=1}^k c_j e^{\zeta_j l}$. Moreover, note that if
we had 0 on the right hand side of $(ii)$, then $f_\infty=f_{opt}$.
$(ii)$ says that the error $\|f_\infty-f_{opt}\|_{\ _\omega^2}$
can be made arbitrarily small relative to the distance
$\|f-f_{opt}\|_{\ _\omega^2}$. Finally, the full theorem in
\cite{altprojection} has a third post, but to define this we need to
discuss angles between manifolds, which we like to avoid. Basically,
the third post says that there exists a number $0<c<1$, whose lower
bound is related to the angle between $\r_{N,k}$ and $\H_{N}$ at $A$, such
that
$$(iii)~ \|Hf_\infty-B_j\|<c^j \|Hf-Hf_{opt}\|.$$ For practical
purposes, this is an important observation, since it says that the
algorithm has so called \emph{$c$-linear convergence}.

Let us now briefly discuss what happens if we are not close enough to $\H_N$ for the above theorem to apply. First of all, we have
never encountered a situation where the algorithm does not converge. Secondly, it is easy to see that both $P_{\H_{N}}$ and $P_{\r_{N,k}}$ are
contractions, so $(B_j)_{j=0}^\infty$ is a bounded sequence. It thus has a convergent subsequence by basic properties of compact sets.
Moreover, it is easy to see that the distance $\|B_{l+1}-B_l\|$ is strictly decreasing with $\l$, and hence the
limit point of the convergent subsequence is in $\H_N$. (However, there is of course no indication that the corresponding $f_{\infty}$
is at all close to $f_{opt}$, so this observation has limit value.) In literature treating similar topics as in this article,
one is usually content with concluding that the algorithm in question has the property that it generates a sequence with a convergent
subsequence having a limit point in the desired set, and attributes this to Zangwill's theorem, \cite{Zangwill}. Clearly, Theorem
\ref{main} provides much more information in our setting; every point in $\H_N$, outside some thin subset, has a neighborhood such
that, if any $B_j$ enters that neighborhood, the sequence $(B_j)_{j=0}^\infty$ will converge. Since the sequence necessarily has more than one accumulation point if it does not converge, the only possibility for divergence is that
$(B_j)_{j=0}^\infty$ wanders back and forth along the valleys of the thin pathological set, between the hills constituting the open
set formed by all nice neighborhoods mentioned above. This seems highly unlikely, but we leave it as an open question
to rule out this possibility. Clearly, it would be interesting to have some concrete values of the parameters $\epsilon$ and $s$ in
Theorem \ref{main}. We will return to this issue in what follows.

Below is an algorithm that specifically describes how to apply the
alternating projection scheme in our case.
\begin{algorithm}~\\
\begin{enumerate}
\item Let $f_0=f$, $l=0$
\item (Application of $P_{\r_{N,k}}$) Compute the first $l$ con-eigenvalues $s_m$ and the con-eigenvectors $u_m$ of $Hf_{l}$ using Theorem \ref{Thm_weight_app}.
The projection $P_{\r_{N,k}} H f_l$ is then given
    \begin{equation} \label{averaging_low_rank_rep}
      \sum_{m=0}^{k} s_m \overline{u_m}u_m^{*}.
    \end{equation}

\item (Application of $P_{\H_{N}}$) Compute
$$
f_{l+1}= H^\ast \left( \sum_{m=1}^k s_m \overline{u_m}
u_m^{*} \right)
$$
\item Increase $l$ and repeat from (2).
\end{enumerate}
\end{algorithm}

\section{The root--MUSIC and ESPRIT methods}
We briefly recapitulate the two most widely used methods for ``high accuracy'' frequency estimation. Our description will follow the
implementation given in \cite{Stoica_moses}.

In a previous section we noted that we can find the nodes for a function $f$ of the form (\ref{eq:Kronecker}), by considering the null
space of a Hankel matrix that is generated from $f$. Recall that it was sufficient to consider a submatrix of size $(k+1)\times(k+1)$ to accomplish this. 
The nodes can in principle be found by finding the roots of the central polynomial, which is the polynomial generated
by a the vector generating $\ker H_{(k+1)}$. 
However, just as discussed previously, this would lead to
numerical instabilities, even when $f$ is a pure a sum of $k$
exponentials. 
From Theorem \ref{th:leftuppercorner_hankel} it is easily seen that we can find the nodes by considering a singular value decomposition of a rectangular Hankel matrix, also generated from
$f$. Let $H^r f \in \m_{2N-M,M}$, with $M>k$, be such a Hankel
matrix. 
and suppose that (\ref{eq:Kronecker}) holds. Then the nodes can in
principle be found by finding the roots of any polynomial generated
by a $u\in\ker H^r f$. Such a $u$ is also in the kernel of
\begin{align} \label{sample_covariance}
(H^r f)^\ast (H^r f) (j,k) = Rf(j,k)
\end{align}
where $0\le j,k < M$.
The matrix $Rf$ is sometimes referred to as the \emph{sample
covariance matrix}. It may seem to be beneficial to work with $R f$
instead of with the full Hankel matrices, since it is in principle
possible to choose $M$ much smaller than $N$  It appears
tractable that we make eigenvalue decomposition on a smaller matrix,
and that the root finding step is also done with smaller matrices.
The standard implementations of root-MUSIC and ESPRIT in \cite{Stoica_moses} work on for instance on $R f$ rather than $H f$.
However, just as discussed previously, a too small $M$ can lead to
numerical instabilities, even when $f$ is purely a sum of $k$
exponentials. Moreover, the matrix $Rf$ needs to be computed. It is
not hard to see that this can be achieved in $\mathcal{O}(N\log N+
M^3)$ time by splitting $H^r$ into two parts and employing FFT. For
large $M$ this is not particularly advantageous. Another drawback is
the loss of precision when forming $(H^r f)^\ast (H^r f)$.

The discussion so far has been conducted under the assumption that
(\ref{eq:Kronecker}) is valid. In the typical situation this is not quite true; the standard assumption is that $f$ contains
additive noise as well. Alternatively, we could be interested in the compression problem of representing a function using only frequencies and coefficients, in which the additive part has more structure than white noise.

Let $$ H_r f=V \Sigma U^\ast.$$ We will as before denote the columns of
$U$ by $u_1,\ldots,u_M$. In the noiseless case, we did see that we
had a great deal of flexibility, as any $u_{m}$, $k<m\le M$ could be
selected to find the nodes. The root-MUSIC method exploits this
property, and tries to use all of the vectors $u_m$, $k<m\le M$ to
reduce the influence of noise. In the root--MUSIC method, roots are
found by solving
$$P_{\mathrm{MUSIC}} (z) = \sum_{m=k+1}^{M}P_{{u_m}}(z)P_{\overline{\check{u}_m}}(z)=0,$$
where $\check{~}$ reverses the order of the elements in a vector.
Loosely speaking, this choice is motivated by the facts that the
roots will appear in pairs when $f$ is a linear combination of
purely oscillatory exponentials. There will be $2M-2$ roots to
$P_{\mathrm{MUSIC}} (z) = 0$. The pairs associate with the true
nodes, will have $\mathrm{Re}(\zeta) \approx 0$, with one slightly
larger than zero and one slightly smaller. For a more detailed
justification on the choice of $P_{\mathrm{MUSIC}}$, cf.
\cite{Stoica_moses}.

In the general case, where there is no constraint on the nodes
$\zeta_p$, $k$ roots need to be selected out of the $2M-2$ that are
given from $P_{\mathrm{MUSIC}} (z) =0$. In the simulations performed
in the later sections, we have used the MUSIC code provided in
\cite{Stoica_moses}, and added a selection step where we approximate
$f$ using all $2 M-2$ nodes using a least squares approach, and then
selecting the $k$ nodes with largest coefficients. It appears
unnecessary to compute nodes that have to be neglected.

The ESPRIT method avoids the step of computing unnecessary nodes.
Instead, a similar approach as in Section \ref{Hankel_properties} is
used. For the noiseless case, it is readily verified that the
eigenvalues of
$$
(U_{(M)}^\ast U_{(M)})^{-1} (U_{(M)}^\ast U_{(1)})
$$
will coincide with the eigenvalues of $e^{\zeta_m}$, $m=1,\dots, k$.
In the ESPRIT method the eigenvalues of the expression above are used
to compute nodes also in the case where noise is present, cf.
\cite{ESPRIT}

We  end this by section by a few remarks about the connection to
autocorrelation and Toeplitz matrices. For a function of the form
(\ref{eq:Kronecker}) where the exponentials are purely harmonic (zero
real part of $\zeta_p$), it holds that
$$
\lim_{N\rightarrow \infty} \frac{1}{N} Rf
$$
is the self-adjoint Toeplitz matrix generated by the autocorrelation
of $f$ (where $f$ is the $4N + 1$-point sampling of a fixed function on a fixed interval). According to a Theorem by Carathéodory \cite{szego},
if the self-adjoint Toeplitz matrix generated by a function has rank
$k$, then that function can be expressed as a sum of $k$ purely
oscillatory exponentials. This motivates alternating projection
schemes between the manifolds of Toeplitz matrices and low rank
matrices, for the approximation of the autocorrelation of a
function. However, the effect of a finite sample length can not be
neglected, and the Toeplitz matrix generated from the
autocorrelation of a pure sum of $k$ oscillatory exponentials, will
fail to have rank $k$. For that sake, the approach we have chosen seems preferable.

\section{Fast algorithms}
There are two operations for which we will need fast numerical methods in the alternating projections approach for frequency detection; (low rank) Takagi decomposition, and the application of the averaging operator (\ref{averaging_low_rank_rep}). It turns out that both operations can be implemented in a fast manner, but the first one will require some more effort than the second.

\begin{proposition} \label{fast_hankel_proposition}
The application of a Hankel matrix to a vector can be done in $\mathcal{O}(N \log(N))$ time by means of FFT.
\end{proposition}
\begin{proof}
This is a standard result \cite{Golub}, and makes use of the fact that circular matrixes are diagonalized by the discrete Fourier transform, and that it easy to construct a circular matrix from a Hankel matrix by permutation and periodic extension. The $\mathcal{O}(N \log(N))$ time complexity can then be achieved by employing FFT.
\end{proof}

\begin{proposition}
The weighted averaging operator $H^\ast$ in (\ref{averaging_low_rank_rep}) can be applied to a rank 1 matrix in $\mathcal{O}(N \log(N))$ time.
\end{proposition}
\begin{proof}
By definition
\begin{align*}
H^\ast(u u^T)(l) &= \frac{1}{\omega(l)} \sum_{j+k=l} w(j) u(j) u(k) w(k),\\ &=\frac{1}{\omega(l)} \sum_{j+k=l} v(j) v(k), \quad
2\le l \le 2N,
\end{align*}
where $v=w u$. It is easy to see that the sum above can we written as a discrete convolution $(v(j)v(l-j))$ using zero padding to avoid boundary effects. The discrete convolutions can then be computed in $\mathcal{O}(N \log(N))$ time using FFT.
\end{proof}

\subsection{Lanczos method for complex symmetric matrices} We will use a modified Lanczos method for finding the $k$ first
con-eigenvalues/con-eigenvectors. The Lanczos method is a way to perform a unitary transformation of a Hermitian matrix
to tridiagonal form, i.e., given $A=A^\ast$, compute $T=Q^\ast A Q$, where $Q$ is unitary and $T=T^\ast$ is tridiagonal. We need a
similar decomposition for complex symmetric matrices. The usage of a modified Lanczos method has been addressed in
\cite{Luk98afast,Fast_SSVD,simoncini_sjostrom}. As we only need to compute the first $k$ con-eigenvalues/con-eigenvectors, we develop a
method customized to that purpose.

The basic step in the Lanczos method is simple. However, it is notorious for the loss of precision, sometimes in a counterintuitive
way. This issue must be addressed carefully. The columns in the unitary matrix $Q$ are computed sequentially, in such a way that
each new column is automatically orthogonal to all previous ones. In practice, finite numerical precision can ruin the orthogonality, and
it can be completely lost in within just a few steps. Two methods that address this are \emph{selective orthogonalization}
\cite{Lanczos_selective} and \emph{partial orthogonalization}
\cite{Lanczos_partial}. We will make use of ingredients from both these methods in our particular setup.

For a given symmetric matrix $A\in\m_{N,N}$, we look for a unitary matrix $Q$, complex numbers $\alpha_{1},\alpha_{2},\ldots$ and nonnegative real numbers $\beta_{1},\beta_{2},\ldots$, such that
\begin{equation}\label{Tridiagonal_form}
T=\overline{Q}AQ^{*},
\end{equation}
where
\begin{equation}\label{apa}
T= \left(
     \begin{array}{ccccccc}
       \alpha_1 & {\beta_1} & 0 & \hdots  & 0 &0\\
       \beta_1 & \alpha_2 & {\beta_2} & \hdots  &0&0 \\
       0 & \beta_2 & \alpha_3 & \ldots & 0  &0\\
       \vdots & \vdots & \vdots & \ddots  &\vdots &\vdots\\
       0 & 0 & 0 & \ldots  &\alpha_{N-1} &{\beta_{N-1}} \\
       0 & 0 & 0 & \hdots  &\beta_{N-1}&\alpha_N \\
     \end{array}
   \right).
\end{equation}
with $\beta_j>0$. The matrices $Q$ and $T$ can be constructed as follows: We want to achieve
$AQ=\overline{Q}T$, which means that
$$Aq_{1}=\alpha_{1}\overline{q_{1}}+\beta_{1}\overline{q_{2}}$$
and, for $2\leq j\leq m-1$,
$$Aq_{j}=\beta_{j-1}\overline{q_{j-1}}+\alpha_{j}\overline{q_{j}}+\beta_{j}\overline{q_{j+1}}$$
and finally
$$Aq_{N}=\beta_{N-1}\overline{q_{N-1}}+\alpha_{N}\overline{q_{N}},$$
where $q_{1},q_{2},\ldots,q_{N}$ are the columns of $Q$.
We choose a unit vector $q_{1}$ and define
$$\alpha_{1}=(Aq_{1},\overline{q_{1}})=q^{T}_{1}Aq_{1},\quad\beta_{1}=\|Aq_{1}-\alpha_{1}\overline{q_{1}}\|,\quad q_{2}=\frac{1}{\beta_{1}}(\overline{Aq_{1}}-\overline{\alpha_{1}}q_{1})$$
if $\beta_{1}\neq 0$, and then, recursively
\begin{align}\label{lanczos_recursion}
\nonumber
\alpha_{j}=(Aq_{j},\overline{q_{j}}),\\
\beta_{j}=\|Aq_{1}-\beta_{j-1}\overline{q_{j-1}}-\alpha_{j}\overline{q_{j}}\|,\\
\nonumber
q_{j+1}=\frac{1}{\beta_{j}}(\overline{Aq_{j}}-\beta_{j-1}q_{j-1}-\overline{\alpha_{j}}q_{j})
\end{align}
as long as $\beta_{j}\neq 0$. One readily verifies, by induction, that the vectors $q_{j}$ are orthonormal.
If, at some step before the last one, $\beta_{m}=0$, then the subspace $X=\mathrm{span}(q_{1},q_{2},\ldots,q_{m})$ of $\C^{N}$ has the property
$$q\in X\Rightarrow \overline{Aq}\in X.$$
For a complete factorization, we can then choose a unit vector in the orthogonal complement in $X$ and proceed. It is easily verified that $\overline{Aq}\in X^{\perp}$ if $q\in X^{\perp}$, so the procedure will eventually yield an orthonormal basis $q_{1},q_{2},\ldots,q_{N}$ for $\C^{N}$, having the desired property. However, as will be discussed in what follows, we will be content with a partial decomposition, and the vanishing of $\beta_m$ at some step generically implies that we do not need to proceed further.
Now let $Q_{m}\in\m_{N,m}$ consist of the first $m$ columns of $Q$ and let $T_{m}$ be the upper left corner $m\times m$-submatrix of $T$.
Write $Q_m=(q_1,\dots,q_m)$ and let $T_m$ denote the $m\times m$
upper left corner submatrix of $T$.
By standard arguments one sees that the con-eigenvectors
of $A_m$ converge to those of $A$, and moreover that for
con-eigenvalues with a low subindex, this convergence is obtained
(within certain precision) with high probability (depending on
$x_0$) for $m<<N$.

The immediate application of the modified Lanczos-method outlined
above is that we can compute con-eigenvector and con-eigenvalues for
$T$ instead of for $A$ (cf. discussion about con-similarity
\cite[p244, p251]{Horn}), which due to the tridiagonal structure of
$T$ it is beneficial. Moreover, since in our setting we are only
interested in the first $k$ con-eigenvectors and $k<<N$, it suffices
to work with $T_m$ for a relatively low number of $m$, increasing
the computational speed. The following lemma makes precise the claim
that the con-eigenvalues of $T_m$ converge to those of $A$.

\begin{lemma} \label{THM_coneigencalues_convergence}
Let $A=A^T$ be given and let $T_m$ be as above. Denote the
con-eigenvectors of $T_m$ by $u_j$ and the corresponding
con-eigenvalues by $\mu_j$, $1\le j \le m$. Then, for each $j$,
there is a con-eigenvalue $\lambda_j$ of $A$ such that
\begin{equation}\label{THM_coneigencalues_convergence_estimate}
|\mu_j - \lambda_j| \le \beta_m |u_j(m)|.
\end{equation}
\end{lemma}
\begin{proof}
Set $e_m=(0,\dots,1)\in\mathbb{R}^{m}$ and note that
$AQ_m-\overline{Q_m}T_m=\beta_m  q_{m+1}e_{m}^T$ by
(\ref{lanczos_recursion}). We apply this to $u_j$ to get
$$
\| A Q_m u_j - \overline{Q}_m T_m u_j \| = \| A Q_m u_j - \mu_j
\overline{Q_m} u_j \| = \|  \beta_m q_{m+1}e_{m}^T u_j\| = \beta_{m}
|u_j(m)|,
$$
and introduce $w_j=Q_m u_j$. Since $\| u_j \|=1$, it follows that $\|w_j\|=1$. Denote the con-eigenvectors of $A$ by $v_l$, and represent $w_j=\sum_l s_l v_l$, $\sum_l |s_l|^2=1$. We then have
\begin{align*}
\| A Q_m u_j - \overline{Q}_m T_m u_j \|^2 &= \| A w_j-\mu_j
\overline{w_j}\|^2 =\| \sum_l s_l (\lambda_l -\mu_j) \overline{v}_l
\|^2 \\ &\ge \min_l |\lambda_l-\mu_j|^2 \sum_l |s_l|^2 =\min_l
|\lambda_l-\mu_j|^2.
\end{align*}
\end{proof}
This is a well known result for the case of Hermitian symmetry, see for instance \cite[p. 69]{Parlett_book}. A similar result is given in \cite[Proposition 2.2]{simoncini_sjostrom}.

Lemma \ref{THM_coneigencalues_convergence} provides a way to control
the convergence of con-eigenvectors. When the quantities in
(\ref{THM_coneigencalues_convergence_estimate}) are small, then
$w_j$ will be a good approximation of the con-eigenvector to $A$
that is associated with $\lambda_j$. In many cases, convergence for
the first con-eigenvalues are reached for comparatively small $m$.
In particular, for the case where the (con)spectrum of $A$ has a
large gap after, say $k$ terms, it is typically only necessary to
use $m$ slightly larger than $k$. This will be the case for all but
the first step in our alternating projection algorithm.

As mentioned before, in a straightforward Lanczos implementation the
orthogonality of $Q$ will quickly be lost due to finite precision
arithmetics. Moreover, and somewhat counterintuitively, the loss of orthogonality will
grow as the con-eigenvectors converge, cf. \cite{Lanczos_selective}.
The simple remedy to this problem is to reorthogonalize $q_{m+1}$ to
all previous $q_j$ at each iteration. However, this increases the
algorithmic complexity of the method. Instead, we want to have a
criterion on when reorthogonalization is needed. The loss of
orthogonalization is also indicative of con-eigenvalue convergence.

Two suggestions on reorthogonalization criteria are given in \cite{Lanczos_partial,Lanczos_selective}. We will follow the approach given in \cite{Lanczos_partial}. Since we are working with con-eigenvalues and con-eigenvectors instead eigenvalues and eigenvectors, we briefly provide the details. 

Due to finite precision arithmetic, we model (\ref{lanczos_recursion}) as
\begin{equation} \label{lanczos_recursion_err}
\beta_{m} \overline{q}_{m+1} = A q_{m}-\overline{q}_m
\alpha_m-\overline{q}_{m-1}{\beta_{m-1}} + \epsilon_m,
\end{equation}
where $\epsilon_m$ describes the error introduced by the finite
precision. We now let $q_j$ denote the vectors \emph{computed} from
the relation (\ref{lanczos_recursion}). Due to the errors
$\epsilon_m$, these vectors will not be orthogonal. Let
$\omega_{j,k}=q_j^\ast q_k$. Then $\omega_{j,k}$ will satisfy the
recursion relation
\begin{align}\label{omega_recursion}
&\omega_{m+1,m+1}=1, \quad \omega_{m+1,m}=q_{m+1}^\ast q_{m} =\psi_{m+1}, \\ \nonumber
&\omega_{m+1,j} = \frac{1}{\beta_m} \Big(\alpha_j \overline{\omega}_{m,j} + \beta_{j-1} \overline{\omega}_{m,j-1}+\beta_j \overline{\omega}_{m,j+1}-\alpha_m \omega_{m,j} -\beta_{m-1} \omega_{m-1,j} \Big) +\vartheta_{m,j},
\end{align}
where $\vartheta_{m,j} = \beta_m^{-1}( q_j^T \epsilon_m-q_m^T \epsilon_j)$.
The last equality follows from multiplying (\ref{lanczos_recursion_err}) by $q_j^T$, and subtracting the same quantity with the indices $j$ and $m$ interchanged. Since $A=A^T$ the quantity $q_j^T A q_m$ then cancels.

Using the recursion formula above, we can monitor the level of lost
orthogonality without explicitly having to compute inner products of
the columns of $Q$. In analogy with the empirical results in
\cite{Lanczos_partial,Lanczos_selective}, we simulate the error
quantities as
\begin{align*}
\vartheta_{m,j} \in {\mathbf N}\left(0,0.3~\varepsilon(\beta_m + \beta_j)\right),\\
\psi_{m+1} \in {\mathbf N}\left(0,0.6~\varepsilon (2N+1)
\frac{\beta_1}{\beta_{m}}\right),
\end{align*}
where ${\mathbf N}(0,\sigma)$ denotes the
complex normal distribution with standard deviation $\sigma$, zero
mean and independent real and imaginary parts. Above, $\varepsilon$
denotes the machine precision.

The maximum loss of precision that can be tolerated without loss of
precision in the coefficients $\alpha_{m+1}$ and $\beta_m$ is
$\sqrt{\varepsilon}$. Once some $\omega_{m+1,j}$ exceeds that level,
it is necessary to reorthogonalize. As seen from (\ref{omega_recursion}), each $\omega_{m+1,j}$ is strongly
influenced by its neighbors. Hence, it will not be efficient to only
reorthogonalize against the vectors $q_j$ where $\omega_{m+1,j}$,
since for isolated $j's$ the orthogonalization would immediately get
lost in the next iteration. Instead, it is beneficial to
reorthogonalize against a batch of $q_j$'s. Hence (and in accordance
with \cite{Lanczos_partial}) we reorthogonalize against the set of
$q_j$ which have $|\omega_{m+1,j}|>\varepsilon^{3/4}$ once
$|\omega_{m+1,j}|>\sqrt{\varepsilon}$ for some $j$.

After a reorthogonalization has taken place, we need to reset the
quantities $\omega_{m+1,j}$. Again following \cite{Lanczos_partial},
we choose $\omega_{m+1,j} \in {\mathbf N}(0,1.5 \varepsilon)$.

The final ingredient is a rule for when to utilize Lemma
\ref{THM_coneigencalues_convergence} for convergence monitoring of
con-eigenvalues. Clearly $m\ge k$ in order to find $k$
con-eigenvalues that have converged. Since the convergence of
con-eigenvalues and the loss of orthogonality are coupled, we
compute a Takagi factorization of $T_m$, once loss of orthogonality
is indicated by $|\omega_{m+1,j}|>\sqrt{\varepsilon}$ for some $j$,
given that $m\le k$. Moreover, we can monitor the behavior of
$\beta_{m}$ to check for convergence. If $\beta_m$ becomes very
small for some $m$, then it means that $Q_m$ defines an almost
invariant (con)subspace under $A$, which implies convergence of the
(non-zero) con-eigenvalues. We let $\varepsilon_L$ denote the desired resolution of con-eigenvalues, and impose the convergence criterion $$\frac{\beta_{m}}{\beta_1}<\varepsilon_L.$$ As always with numerical
implementations, it can be difficult to determine how small
$\beta_m$ has to be in order to consider it to have almost vanished, i.e., if $\varepsilon_L$ is chosen very small. A typical feature of this case is
that the last value $\beta$ jumps dramatically in size. This behavior
also serves as a good criterion for when to check for convergence by
means of Lemma \ref{THM_coneigencalues_convergence}.

In the procedure above, we need to compute the Takagi factorization of $T_m$. The cost of that step when using Proposition \ref{Takagi_real_compute} is $\mathcal{O}(m^3)$. However, due to the tridiagonal structure there are methods to compute this in $\mathcal{O}(m^2)$ time, cf. \cite{Luk98afast,twisted_takagi,Fast_SSVD}. These methods are based on straightforward modifications of methods for eigenvalue decomposition of tridiagonal Hermitian matrices.

The most expensive step in the Lanczos procedure described above is the matrix vector multiplication $A q_m$ in (\ref{lanczos_recursion}). However, this step can be computed in $\mathcal{O}(N \log N)$ time by Proposition \ref{fast_hankel_proposition}.
\begin{proposition}
The time complexity for computing the first $k$ con-eigenvectors and
con-eigenvalues of a Hankel matrix to accuracy $\varepsilon$ using
the modified Lanczos method described above is $\mathcal{O}( m N\log
N + m^2)$, where $m$ denotes the total number Lanczos steps, and
where $m\ge k$, but where $m$ is typically of the same order as $k$.
\end{proposition}

\thispagestyle{empty}
\section{Numerical simulations}
\subsection{Performance analysis}
In this section we compare the performance of our approach against the ESPRIT and root-MUSIC methods. We simulate functions of the form $f= f_0 + n$, where
$$
f_0(l) = \sum_{p=1}^{k} c_p e^{\zeta_p l},
$$
and where $n$ is a noise component. The coefficients $c_p$ are chosen as complex normal distributed variables, and the nodes as $\zeta_p= 1/(4N+1) (50 Z^r_p + i Z^i_p)$, where $Z^r_p$ and $Z^i_p$ are normally distributed.

The noise component is constructed by letting $\tilde{n}(k)= n_r(k) + i n_i(k)$, where $n_r$ and $n_i$ are normally distributed noise, and where
$$
n = \sqrt{ \frac{\| f_0\|^2}{\| \tilde{n}\|^2} 10^{-\mathrm{SNR}/10}} \tilde{n},
$$
for some \emph{signal to noise} parameter $SNR$. By this
construction, the signal to noise ratio will be exactly equal to the
parameter $SNR$ when measured in dB. Throughout the tests, we have
chosen to work with a signal length of 511, i.e. $N=256$. This is
chosen to make the FFT routines run fast. All simulations have
been run in a MATLAB environment, without any compiled
optimizations. For the ESPRIT and root-MUSIC, we have used the
routines provided in \cite{Stoica_moses}, with minor modifications
to make them work for the case $\mathrm{Re}(\zeta_p) \ne 0$.
The accuracy parameter $\epsilon$ used in the alternating projection
method has been chosen to be a factor 100 lower than the noise
magnitude.

In figure \ref{first_figure}, we show some simulation results for
the different methods.  We conduct a small number of simulations for
$SNR=10~dB$ and $k=10$, and consider the performance in terms of the
errors generated by the different methods.
\begin{figure}[ht]
\centering
\includegraphics[width=0.48\linewidth]{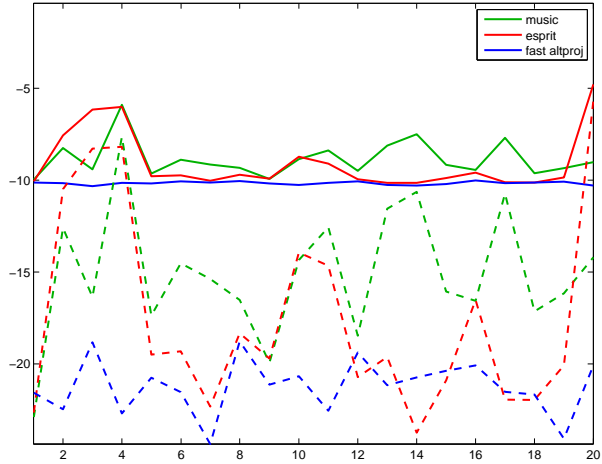}
\caption{The solid curves show the error $\| g-f\|/\|f\|$ with $g$
obtained with MUSIC (green), ESPRIT (red) and our propose method
(blue), respectively,  for SNR=10. The dashed ones show the
counterparts for  $\| g-f_0\|/\|f_0\|$.} \label{first_figure}
\end{figure}
We display the errors are displayed in two ways; in relation to the pure signal $f$ and in relation to the noise one $f_0$.

We see that our proposed method systematically has a smaller error in
both ways of measurement. We also note that for all methods we have
a substantially smaller error when compared to the pure signal $f_0$
instead of the noisy one. Hence, all three methods successfully
filter out a large part of the noise. It is also notable how close
the error in relation to $f$ is to the signal to noise ratio for our
proposed method. This is also implied by Theorem \ref{main}.
Basically, in the notation of Section \ref{altproj}, we have
$g=g_0$, and it is reasonable to assume that $f_0\approx g_{opt}$.
This is because the noise has a high probability of being orthogonal
to $f_0$, and that $P_{\M^n\cap\H}$ locally acts as an orthogonal
projection, (which is further elaborated on in \cite{altprojection}). Thus,
Figure \ref{first_figure} can be interpreted as the upper blue line
shows $\|f-g_{opt}\|/\|f\|$, whereas the lower blue line gives an
indication of the size of $\|g_\infty-g_{opt}\|/\|f\|$. In terms of
Theorem \ref{main} with $A=Hf_0$ of norm 1 and $s=0.1$, this means
that we can pick $\epsilon$ around $0.1$ as well. Although the above
images are constructed using standard $\ell^2$-norm, not the weighted one
required for Theorem \ref{main} to kick in, it is interesting to
observe that this is in line with the observations in \cite{altprojection}.
There, using more carefully conducted examples to test Theorem
\ref{main}, it seems that one can take $s\approx\epsilon$ when
working with $k\approx 10$.

It is interesting to see how these result depend on the different
parameters, i.e., the number of nodes $k$ and the noise level $SNR$.
In Figure \ref{fig:n_dependence} we have conducted more thorough
investigations. For each $k=1, 2, \dots, 30$ we have done 100
simulations and computed average results. The averaging has been
made in dB, in order to limit the effect of outlier results.
\begin{figure}[ht]
\centering
\includegraphics[width=0.48\linewidth]{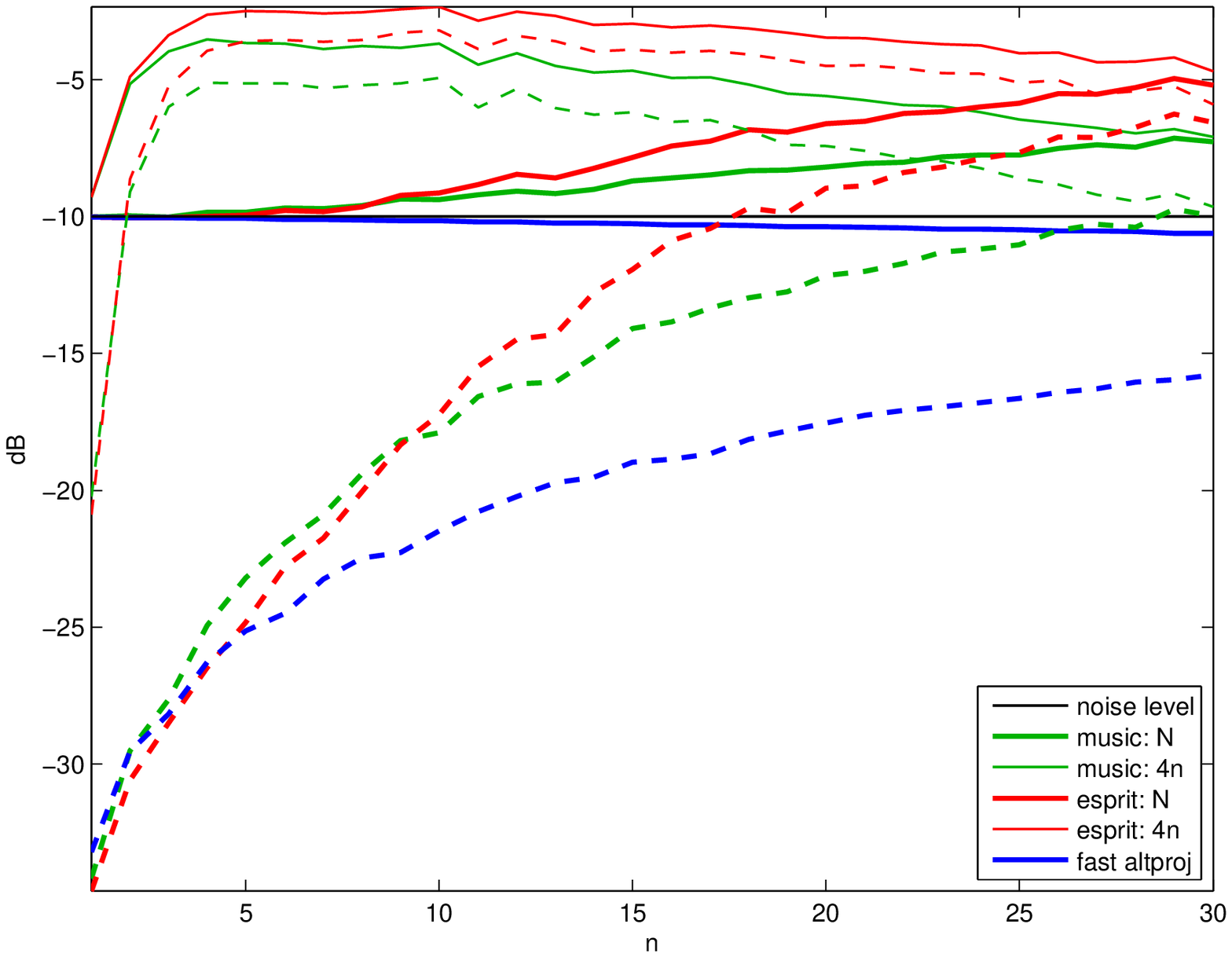}
\includegraphics[width=0.48\linewidth]{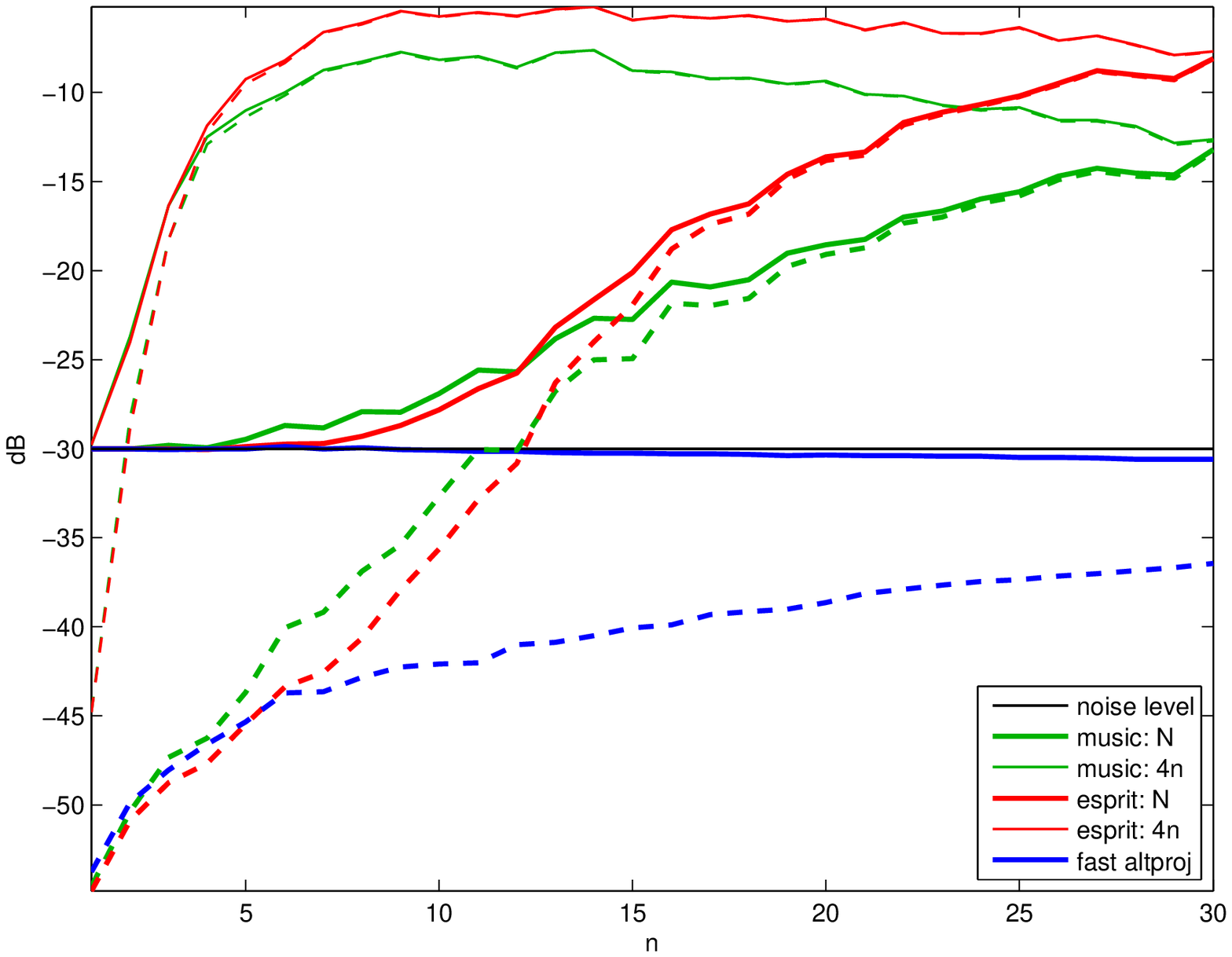}
\caption{The left panel shows results for SNR=10, and the right one results for SNR=30. The solid curves show the dependence on $k$ of the average error $\| g-f\|/\|f\|$ with g obtained with MUSIC (green), ESPRIT (red) and our propose method (blue), respectively, for SNR=10 . The dashed ones show the counterparts for  $\| g-f_0\|/\|f_0\|$. The average for each $k$ is done over 100 simulations. The thin lines show the results obtained for MUSIC and ESPRIT for $M = 4k $, whereas the thick lines show results for $M=N$. }
\label{fig:n_dependence}
\end{figure}
As for Figure \ref{first_figure}, we display errors in two ways,
using solid lines for errors in comparison to the noise signal $f$
and dashed lines for the comparison to the original one, $f_0$. We
also display some of the impact that the choice of size ($M$) of $R$ in
(\ref{sample_covariance}) has. The thin lines in red and green show
errors for $M=4k$ and the thick lines show the counterpart for
$M=N=256$.

There are a few interesting conclusions that can be drawn from the
results depicted in Figure \ref{fig:n_dependence}. First, we note
that for the cases were $k$ is small, all three methods perform
comparably well, given that the size ($M$) of the sample covariance
matrix used in MUSIC and ESPRIT is sufficiently large. However, as
$k$ increases, the alternating projection method starts to outperfom
the other two. We can again note that the errors (compared to $f$)
produced by the alternating projection method almost coincide with
the signal to noise ratio. Moreover, in terms of Theorem
\ref{main}, Figure \ref{fig:n_dependence} seems to indicate that
$\epsilon\approx s$ is a good rule of thumb, although the ratio gets
slightly worse as the complexity of the manifold $\M^n\cap\H$
increases with increasing $k$.

From the results we have seen so far we can conclude that the
alternating projection method should be the method of choice unless
$k$ is very small, given that the prime concern is to minimize the
estimation errors. The other criterion for method selection is speed.
The computational times for the different methods is displayed in
Figure \ref{fig_time}.
\begin{figure}[ht]
\centering
\includegraphics[width=0.48\linewidth]{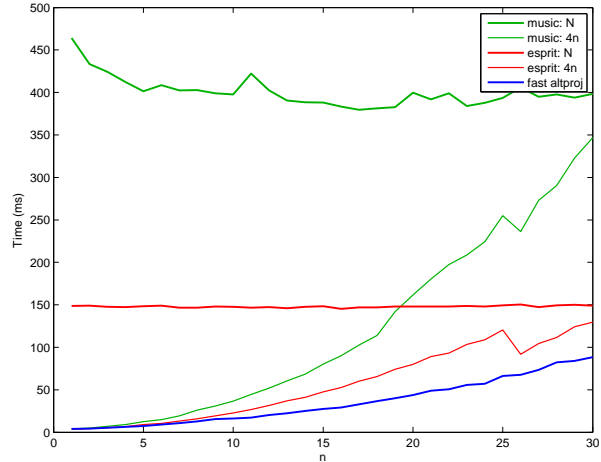}
\caption{Average execution time for the different methods in milliseconds. The line notation as in Figure \ref{fig:n_dependence} is used. \label{fig:time}}
\end{figure}
As mentioned before, the MUSIC and ESPRIT algorithms that are used
are slightly modified versions of the ones given in
\cite{Stoica_moses}. In Figure \ref{fig:time} the fast alternating
projection method is the fastest. The MUSIC and ESPRIT algorithms
are substantially slower for high $M$. For the MUSIC algorithm, the
most time consuming step is the root solving step. We note, however,
that by using our fast method for finding the first $k$
con-eigenvector / con-eigenvalues, we can construct a method that
would have much resemblance with the ESPRIT method as described in
\cite{ESPRIT}. It seems advantageous to work directly with
Hankel matrices rather than the covariance sample matrix of
(\ref{sample_covariance}). Using such an approach, we would be able
to construct a computational method that would provide similar
results as the ESPRIT algorithm described in \cite{Stoica_moses},
but substantially faster than the one based on the sample covariance
matrix. From the results in Figure \ref{fig:n_dependence} we can
conclude that the results would not be as good as the ones obtained
by the fast alternating projection method proposed here. However, it
would be faster, as it would only involve one decomposition step. In
other words, it would be equivalent to using the alternating
projection scheme with only one iteration.

A natural question would then be how much faster such ``fast''
ESPRIT algorithm would be. A first guess would be that the speed
ratio would be proportional to the number of alternating projections
performed before the target accuracy $\epsilon$ is reached. It turns out
that the fast alternating projection method is faster than that. The
reason for this is that fewer Lanczos iterations are required in
each alternating projection iteration.
\begin{figure}[ht]
\centering
\includegraphics[width=0.48\linewidth]{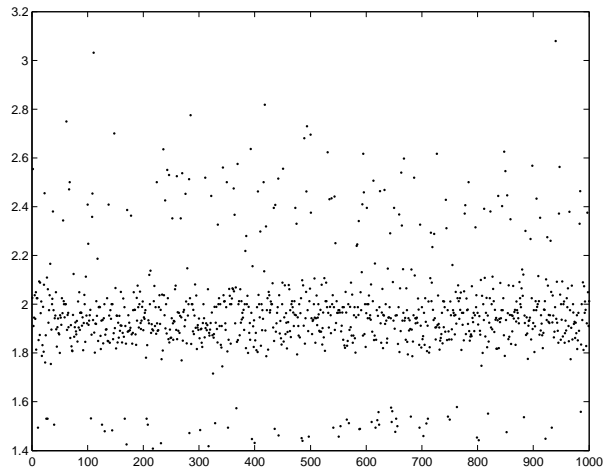}
\caption{The ratio between the time for total number of iterations
compared to the first one for the alternating projection method for
1000 simulations}. \label{fig_time}
\end{figure}
In Figure \ref{fig_time} we display the ratio between the
total time and the time for the first iteration for the alternating
projection method. We see that the ratio typically lies around 2.
This means that the fast alternating projection method would only be
about twice as expensive as a fast implementation of ESPRIT, while
providing smaller errors. Again, we note that the proposed fast
alternating projection method is substantially faster than the
standard implementation of ESPRIT and MUSIC.

\subsection{Approximations with Gaussians}
As a final example, we show some results concerning the
approximation of functions with Gaussians with fixed half-width,
using the fast alternating projection method with Gaussian weights.
There are two possible interesting cases. The first one concerns the
case where the functions are of the form
\begin{equation} \label{gauss_rep}
\sum_p c_p e^{-\alpha (x-x_p)^2 + i \xi_p x},
\end{equation}
with fixed (and known) constant $\alpha$. The second case concerns the approximation of functions using a Gaussian window, for example as done in time--frequency analysis. Using a non-linear approach may be beneficial compared to short-time Fourier transform representations with overlapping windows. However, we will in this section only show some results concerning (\ref{gauss_rep}).

In Figure \ref{fig_gauss} we show the result from one simulation using a function of the form (\ref{gauss_rep}), using $10$ Gaussians. In order to approximate this function using exponentials, we choose the weights $w$ such that $\sqrt{\omega_l}$ approximates $e^{-\alpha l/2N}$, $l=-2N,\dots, 2N$. To this end, we choose
$$
w_k = \sqrt{\frac{2}{4N+1}  \sqrt{ \frac{8 \alpha}{\pi}}}e^{-4 \alpha k/{4N}}, \quad k=-N,\dots,N.
$$
For sufficiently narrow Gaussians (large $\alpha$), we will then have that $\sqrt{\omega_l} \approx e^{-\alpha l/2N}$.
\begin{figure}[ht]
\centering
\includegraphics[width=0.68\linewidth]{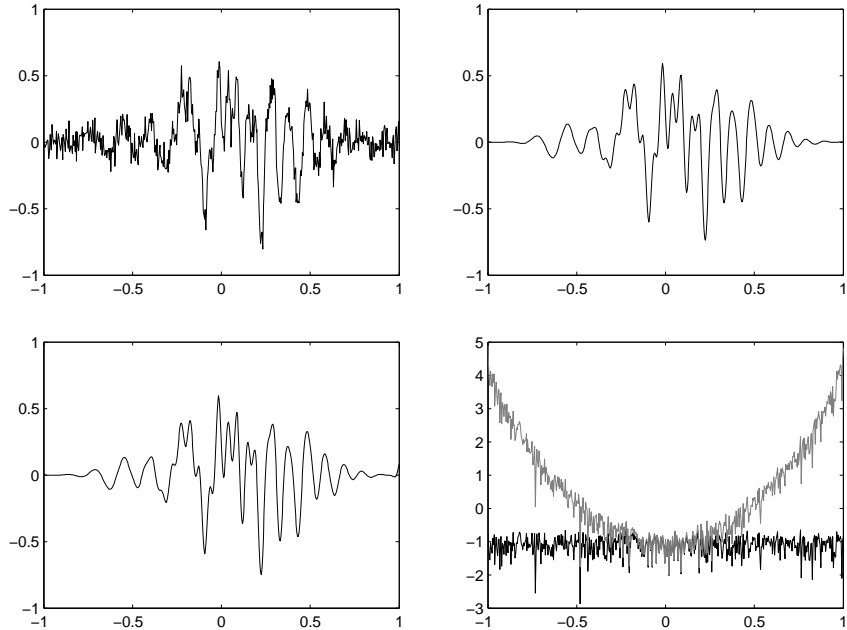}
\caption{In the top left panel the noisy signal $f$ is shown, the top right shows the original $f_0$, the bottom left shows the reconstruction, and the bottom right shows errors. The errors are shown unscaled in gray, and scaled with respect to $\sqrt{\omega}$ in black.}.
\label{fig_gauss}
\end{figure}
Just as before we let $f_0$ be of the form  (\ref{gauss_rep}) and use additive noise to obtain $f$. One simulation is shown in Figure \ref{fig_gauss} for $SNR=10$. Before we start the alternating projection scheme, we divide $f$ pointwise with $1/\sqrt{\omega_l}$. This will boost the amplitude at the endpoints of $f$ substantially, but since we approximate using $\omega_l$ as a weight, we will obtain a uniform approximation. The noise will, however, not be uniform with this approach, but larger at the end-points.

The result from one simulation is shown in Figure \ref{fig_gauss}. The noise signal is depicted in the top left panel, while the original is displayed in the top right panel. In the bottom left we see the obtained reconstruction. We can see that most of the features from the original signal is captured. In the bottom right panel we show pointwise errors; in black the error weighted with $\omega_l$ and in grey the unweighted pointwise error.

\section{Conclusions}
We have developed a method for the fast estimation of complex frequencies using an alternating projection scheme between Hankel matrices and rank $k$ Takagi representations. The method has a time complexity of $\mathcal{O}(n N\log N + n^3)$. FFT routines are used both to get fast matrix--vector multiplications, and to project rank $k$ representations to Hankel matrices. In order to compute the first $k$ Takagi vectors, we employ a modified Lanczos scheme for self-adjoint matrices.
The number of necessary alternating projection steps depends on an accuracy parameter, but in typical situations the total time is only twice as large as the time needed for the first iteration. The reason for this is that fewer Lanczos steps are needed when the matrix we obtain is closer to being both Hankel, and rank $k$.

In our simulations we see that the proposed method performs better both with regards to speed and approximation accuracy, compared to standard implementations like root-MUSIC and ESPRIT. We also verify that the errors that we obtain behave in the manner theoretically predicted in \cite{altprojection}. The method works also for some weighted representations. A particular case of weights that can be used are Gaussian weights, for which case some numerical examples are provided.

\section{Acknowledgements}
This work was supported by the Swedish Research Council and the Swedish Foundation for International Cooperation in Research and Higher Education.

\section{Appendix; the set of tangential points in $\H_k$ is thin}\label{appendix}
Following the terminology of \cite{altprojection}, a point $A\in\H_k$ is called
regular if the dimension of $\r_k$ $\H_k$ and $\H$ are constant in a
neighborhood of $A$. Thus theorem \ref{tKronecker} says that the set
of non-regular points is thin. Moreover, recall that a point $A\in
\H_k$ is called non-tangential if \begin{equation}\label{eq_angle1}
T_{\r_k}(A)\cap T_{\H}(A)=T_{\r_k\cap\H}.\end{equation} In order to
prove Theorem \ref{main}, we need to show that the set of tangential
points in $\H_k$ is thin, and then apply Theorem 6.1 of
\cite{altprojection}.

\begin{theorem}\label{t3}
The set of tangential points is thin in $\H_k$.
\end{theorem}

\begin{proof} By Theorem \ref{tKronecker} we immediately get that all points in
$\H_k^n$ are regular and that $\H\setminus\H_k^n$ is thin. To verify
that $A\in\H_k^n$ is non-tangential, it thus suffices to establish
(\ref{eq_angle1}), e.g. that $T_{\H}(A)\cap T_{\r_k^d}(A) =
T_{\H_k^n}(A)$, since $\H_k^n\subset\r_k^d$. Clearly
\begin{equation}\label{fis} T_{\H}(A)\cap T_{\r_k^d}(A) \supset
T_{\H_k^n}(A).\end{equation} By Theorem \ref{tKronecker} and the
fact that $A$ is regular we have $\dim(T_{\H_k^n}(A))=4k$ and
\begin{align*} &\dim(T_{\H}(A)\cap T_{\r_k^d}(A))=\dim(T_{\H}(A))+\dim( T_{\r_k^d}(A))-\dim(T_{\H}(A)+
T_{\r_k^d}(A))=\\&= 2(2n-1)+2(2kn-k^2)-\dim(T_{\H}(A)+
T_{\r_k^d}(A)).\end{align*} To establish the reverse inclusion to
(\ref{fis}), it thus suffices to show that $\dim(T_{\H}(A)\cap
T_{\r_k^d}(A))\leq 4k$, or equivalently
\begin{equation*} \dim(T_{\H}(A)+T_{\r_k^d}(A))\geq 2(2n-1)+2(2kn-k^2)-4k.\end{equation*}
Moreover, since both subspaces are closed under multiplication by
$\C$, it suffices to verify \begin{equation}\label{mut}
\dim_\C(T_{\H}(A)+T_{\r_k^d}(A))\geq 2n-1+2kn-k^2-2k,\end{equation}
where $\dim_\C$ denotes the dimension over $\C$. To this end, note
that the map $\mathfrak{W}:(\m_{n,k})^2\rightarrow\m_{n,n}$ given by
$$\mathfrak{W}(U,V)=VU^*=\sum_{j=1}^k {v}_j{u}_j^*,$$ (where ${u}_j,~{v}_j$ denote the columns of $U$ and $V$
respectively), is an immersion onto $\r_k$. By this we mean that for
each $A\in\r_k$ there exists $U_A,~V_A$ such that
$A=\mathfrak{W}(U_A,V_A)$ and, if $A\in\r_k^d$, then
$$T_{\r_k^d}(A)=\Ran \partial\mathfrak{W},$$
where $\partial\mathfrak{W}$ denotes the derivative of
$\mathfrak{W}$. In this section we define
$\mathfrak{U}:\C^n\rightarrow\m_{n,k}$ and
$\mathfrak{V}:\C^n\times\C^n\rightarrow\m_{n,k}$ via
$$\mathfrak{U}(\alpha)=\left(
                         \begin{array}{ccc}
                           1 & \cdots & 1 \\
                           \alpha_1 & \cdots & \alpha_k \\
                           \alpha_1^2 & \cdots & \alpha_k^2 \\
                           \vdots & \vdots & \vdots \\
                           \alpha_1^n & \cdots & \alpha_k^n\\
                         \end{array}
                       \right),\quad\quad \mathfrak{V}(c,\alpha)=\left(
                         \begin{array}{ccc}
                           c_1 & \cdots & c_k \\
                           c_1\alpha_1 & \cdots & c_k\alpha_k \\
                           c_1\alpha_1^2 & \cdots & c_k\alpha_k^2 \\
                           \vdots & \vdots & \vdots \\
                           c_1\alpha_1^n & \cdots & c_k\alpha_k^n\\
                         \end{array}
                       \right).
$$
It is easily seen that, given any $\alpha\in\C$, the matrix
\begin{equation}\label{hankelrank} H(\alpha)=\left(
                                            \begin{array}{ccccc}
                                              1 & \alpha & \alpha^2 & \cdots & \alpha^{N-1} \\
                                              \alpha & \alpha^2 & \iddots & \alpha^{N-1} & \alpha^N \\
                                              \alpha^2 & \iddots & \iddots & \iddots & \vdots \\
                                              \vdots & \alpha^{N-1} & \iddots & \iddots & \alpha^{2N-3} \\
                                              \alpha^{N-1} & \alpha^N & \cdots & \alpha^{2N-3} & \alpha^{2N-2} \\
                                            \end{array}
                                          \right)
\end{equation}
defines a rank 1 Hankel matrix. Thus
\begin{equation}\label{H}\mathfrak{H}(c,\alpha)=\sum_{j=1}^k c_jH(\alpha_j)\end{equation} is a rank $k$ Hankel matrix. It is clear that
$$\mathfrak{H}(c,\alpha)=\mathfrak{W}(\mathfrak{U}(\alpha),\mathfrak{V}(c,\alpha)).$$
 Thus, whenever $A=\mathfrak{H}(c,\alpha)\in\H_k^n$, we have
\begin{equation}\label{apanson} T_{\r_k^d}(A)=\Ran
\partial\mathfrak{W}(\mathfrak{U}(\alpha),\mathfrak{V}(c,\alpha)).\end{equation}
Now, it is not hard to see that
$\partial\mathfrak{W}(\mathfrak{U}(\alpha),\mathfrak{V}(c,\alpha))$
is a polynomial in the variables $c$ and $\alpha$. To visualize, say
that $n=3$ and $k=2$. Then the right hand side is given as the span
of the 12 matrices  $$\left(
                   \begin{array}{ccc}
                     1 & \alpha_j & \alpha_j^2 \\
                     0 & 0 & 0 \\
                     0 & 0 & 0 \\
                   \end{array}
                 \right),\left(
                   \begin{array}{ccc}
                     0 & 0 & 0 \\
                     1 & \alpha_j & \alpha_j^2 \\
                     0 & 0 & 0 \\
                   \end{array}\right),\left(
                   \begin{array}{ccc}
                     0 & 0 & 0 \\
                                          0 & 0 & 0 \\
                     1 & \alpha_j & \alpha_j^2 \\
                   \end{array}
                 \right),\quad\quad (j=1,2),$$ and $$\left(
                                   \begin{array}{ccc}
                                     c_j & 0 & 0 \\
                                     c_j\alpha_j & 0 & 0 \\
                                     c_j\alpha_j^2 & 0 & 0 \\
                                   \end{array}
                                 \right),\left(
                                   \begin{array}{ccc}
                                     0 & c_j & 0 \\
                                     0 & c_j\alpha_j & 0 \\
                                     0 & c_j\alpha_j^2 & 0 \\
                                   \end{array}
                                 \right),\left(
                                   \begin{array}{ccc}
                                     0 & 0 & c_j \\
                                     0 & 0 & c_j\alpha_j \\
                                     0 & 0 & c_j\alpha_j^2\\
                                   \end{array}
                                 \right),\quad\quad (j=1,2).$$ Moreover, picking a basis for $\m_{n,n}$ (for example, the
standard one which we order lexicographically), the right hand side
of (\ref{apanson}) can be identified with the range of a matrix with
polynomial entries which we denote by
$\widetilde{\partial\mathfrak{W}}(\mathfrak{U}(\alpha),\mathfrak{V}(c,\alpha))$.
To continue the example, we get
\begin{equation}\label{muta} \widetilde{\partial\mathfrak{W}}(\cdots)=
\left(
  \begin{array}{cccccccccccc}
    1 & 0 & 0 & 1 & 0 & 0 &                   c_1 & 0 & 0 & c_2 & 0 & 0 \\
    0 & 1 & 0 & 0 & 1 & 0 &                   c_1\alpha_1 & 0 & 0 & c_2\alpha_2 & 0 & 0 \\
    0 & 0 & 1 & 0 & 0 & 1 &                   c_1\alpha_1^2 & 0 & 0 & c_2\alpha_2^2 & 0 & 0 \\
    \alpha_1 & 0 & 0 & \alpha_2 & 0 & 0 &     0 & c_1 & 0 & 0 & c_2 & 0 \\
    0 & \alpha_1 & 0 & 0 & \alpha_2 & 0 &     0 & c_1\alpha_1 & 0 & 0 & c_2\alpha_2 & 0 \\
    0 & 0 & \alpha_1 & 0 & 0 & \alpha_2 &     0 & c_1\alpha_1^2 & 0 & 0 & c_2\alpha_2^2 & 0 \\
    \alpha_1^2 & 0 & 0 & \alpha_2^2 & 0 & 0 & 0 & 0 & c_1 & 0 & 0 & c_2 \\
    0 & \alpha_1^2 & 0 & 0 & \alpha_2^2 & 0 & 0 & 0 & c_1\alpha_1 & 0 & 0 & c_2\alpha_2 \\
    0 & 0 & \alpha_1^2 & 0 & 0 & \alpha_2^2 & 0 & 0 & c_1\alpha_1^2 & 0 & 0 & c_2\alpha_2^2 \\
  \end{array}
\right)
\end{equation}
With \begin{equation*} E_j(m,l) = \left\{
          \begin{array}{ll}
            1, & \hbox{if $m+l=j$;} \\
            0, & \hbox{otherwise.}
          \end{array}
        \right.,
\end{equation*}
$\tilde{\H}$ is spanned by
$\widetilde{E}_1,\ldots,\widetilde{E}_{2n-1}$, where the notation is
self-explanatory. In our example we get
\begin{equation}\label{muta1}\tilde{\H}=\Ran \left(
                              \begin{array}{ccccc}
                                1 & 0 & 0 & 0 & 0 \\
                                0 & 1 & 0 & 0 & 0 \\
                                0 & 0 & 1 & 0 & 0 \\
                                0 & 1 & 0 & 0 & 0 \\
                                0 & 0 & 1 & 0 & 0 \\
                                0 & 0 & 0 & 1 & 0 \\
                                0 & 0 & 1 & 0 & 0 \\
                                0 & 0 & 0 & 1 & 0 \\
                                0 & 0 & 0 & 0 & 1 \\
                              \end{array}
                            \right)
\end{equation}
Let us denote the matrix obtained by adjoining
$\widetilde{\partial\mathfrak{W}}(\cdots)$ and $\widetilde{\H}$ by
$[\widetilde{\partial\mathfrak{W}}~\tilde{\H}]$. To verify
(\ref{mut}), it thus suffices to show that
\begin{equation}\label{flis}
\Rank[\widetilde{\partial\mathfrak{W}}~\widetilde{\H}] \geq
2n-1+2kn-k^2-2k,
\end{equation} holds, evaluated at $(\mathfrak{U}(\alpha),\mathfrak{V}(c,\alpha))$ for some $c,\alpha$ such that $\mathfrak{H}(c,\alpha)\in\H_k^n.$
Note that\begin{itemize}\item[$(i)$] If we can establish
(\ref{flis}) for one point $A=\mathfrak{H}(c,\alpha)$, then it
easily follows that (\ref{flis}) holds at all but a thin set of
points $A$. To see this, set $q=2n-1+2kn-k^2-2k$ and first note that
we can pick a $q\times q$ submatrix of
$[\widetilde{\partial\mathfrak{W}}~\tilde{\H}]$ whose determinant is
a non-zero polynomial. Thus, by standard algebraic geometry, the set
of points $(c,\alpha)$ where the determinant is zero is thin in
$\C^{2k}$. Finally, it is also clear that the image of a thin set
under a chart, in this case $\mathfrak{H}$, is again thin.
\item[$(ii)$] Let $B=\mathfrak{W}(U_B,V_B)\in\H_k$ be a point such that \begin{equation}
\label{g}\Rank{\partial\mathfrak{W}}(U_B,V_B)=2kn-k^2,\end{equation}
but where $(U_B,V_B)$ is not necessarily in the closure of the range
of $(\mathfrak{U},\mathfrak{V})$. We claim that in order to
establish $(i)$, it suffices to establish (\ref{flis}) at the point
$(U_B,V_B)$. To see this, first note that by (\ref{g}), $\r_k$ is
locally a manifold (of dimension $2kn-k^2$) around $B$, and we can
take an affine subspace of $\n\subset\m_{n,k}^2$ containing
$(U_B,V_B)$ such that $\mathfrak{W}|_{\n}$ becomes a local chart for
$\r_k$. If (\ref{flis}) holds for $(U_B,V_B)$, then arguing as above
with determinants, it holds in a neighborhood of $(U_B,V_B)$. By
Theorem \ref{tKronecker}, $\H_k^n$ is dense in $\H_k$, so in
particular we can pick a $C\in\H_k^n$ and corresponding
$U_C,~V_C\in\n$ and $~c_C, \alpha_C \in\C^n$ such that
$C=\mathfrak{W}(U_C,V_C)=\mathfrak{H}(c_C, \alpha_C)$ and
(\ref{flis}) is satisfied for
$[\widetilde{\partial\mathfrak{W}}(U_C,V_C)~\tilde{\H}]$. By
(\ref{apanson}) and (\ref{g}) we have
$$\Ran{\partial\mathfrak{W}}(U_C,V_C)=T_{\r_k^d}(C)=\Ran{\partial\mathfrak{W}}
(\mathfrak{U}(\alpha_C),\mathfrak{V}(c_C,\alpha_C)),$$ which shows
that (\ref{flis}) is satisfied at
$(\mathfrak{U}(\alpha_C),\mathfrak{V}(c_C,\alpha_C))$, as desired.
\end{itemize}
So, it remains to verify (\ref{flis}) and (\ref{g}) for some point
$B=\mathfrak{W}(U_B,V_B)\in\H_k$. In terms of our example, we pick
$$U_B=\left(
                                                \begin{array}{cc}
                                                  1 & 0 \\
                                                  0 & 1 \\
                                                  0 & 0 \\
                                                \end{array}
                                              \right),\quad \quad V_B=\left(
                                                \begin{array}{cc}
                                                  1 & 1 \\
                                                  1 & 0 \\
                                                  0 & 0 \\
                                                \end{array}
                                              \right)
$$
so that $B$ becomes the rank 2 Hankel operator $$B=\left(
                                                     \begin{array}{ccc}
                                                       1 & 1 & 0 \\
                                                       1 & 0 & 0 \\
                                                       0 & 0 & 0 \\
                                                     \end{array}
                                                   \right).
$$
Then $\partial\mathfrak{W}(U_B,V_B)$ is spanned by the 6
"$V$-derivatives";
$$\left(
                   \begin{array}{ccc}
                     1 & 0 & 0 \\
                     0 & 0 & 0 \\
                     0 & 0 & 0 \\
                   \end{array}
                 \right),\left(
                   \begin{array}{ccc}
                     0 & 0 & 0 \\
                     1 & 0 & 0 \\
                     0 & 0 & 0 \\
                   \end{array}
                 \right),\left(
                   \begin{array}{ccc}
                     0 & 0 & 0 \\
                     0 & 0 & 0 \\
                     1 & 0 & 0 \\
                   \end{array}
                 \right),\left(
                   \begin{array}{ccc}
                     0 & 1 & 0 \\
                     0 & 0 & 0 \\
                     0 & 0 & 0 \\
                   \end{array}
                 \right),\left(
                   \begin{array}{ccc}
                     0 & 0 & 0 \\
                     0 & 1 & 0 \\
                     0 & 0 & 0 \\
                   \end{array}
                 \right),\left(
                   \begin{array}{ccc}
                     0 & 0 & 0 \\
                     0 & 0 & 0 \\
                     0 & 1 & 0 \\
                   \end{array}
                 \right)$$
                 and the 6 "$U$-derivatives;
$$\left(
                   \begin{array}{ccc}
                     1 & 0 & 0 \\
                     1 & 0 & 0 \\
                     0 & 0 & 0 \\
                   \end{array}
                 \right),\left(
                   \begin{array}{ccc}
                     0 & 1 & 0 \\
                     0 & 1 & 0 \\
                     0 & 0 & 0 \\
                   \end{array}
                 \right),\left(
                   \begin{array}{ccc}
                     0 & 0 & 1 \\
                     0 & 0 & 1 \\
                     0 & 0 & 0 \\
                   \end{array}
                 \right),\left(
                   \begin{array}{ccc}
                     1 & 0 & 0 \\
                     0 & 0 & 0 \\
                     0 & 0 & 0 \\
                   \end{array}
                 \right),\left(
                   \begin{array}{ccc}
                     0 & 1 & 0 \\
                     0 & 0 & 0 \\
                     0 & 0 & 0 \\
                   \end{array}
                 \right),\left(
                   \begin{array}{ccc}
                     0 & 0 & 1 \\
                     0 & 0 & 0 \\
                     0 & 0 & 0 \\
                   \end{array}
                 \right)$$
This is clearly an 8-dimensional space not including $$\left(
                   \begin{array}{ccc}
                     0 & 0 & 0 \\
                     0 & 0 & 0 \\
                     0 & 0 & 1 \\
                   \end{array}
                 \right),$$ which happens to be $E_5$ in the basis for $\H$, and thus
\begin{equation}\label{flis1}
\Rank[\widetilde{\partial\mathfrak{W}}(U_B,V_B)~\widetilde{\H}]=9=
2*3-1+2*2*3-2^2-2*2,
\end{equation}
establishing (\ref{flis}) in this particular case.
\begin{figure}[tbh!]\label{fig7} \unitlength=1in
\begin{center}
\psfrag{E1}[][l]{$E_1$} \psfrag{E2}[][l]{$E_2$}
\psfrag{E2k-1}[][l]{$E_{2k-1}$} \psfrag{E2k}[][l]{$E_{2k}$}
\psfrag{E2k+1}[][l]{$E_{2k+1}$} \psfrag{En}[][l]{$E_{n}$}
\psfrag{E2n-1}[][l]{$E_{2n-1}$} \psfrag{nn}[][l]{$n$}
\psfrag{kk}[][l]{$k$}
\includegraphics[width=3.5in]{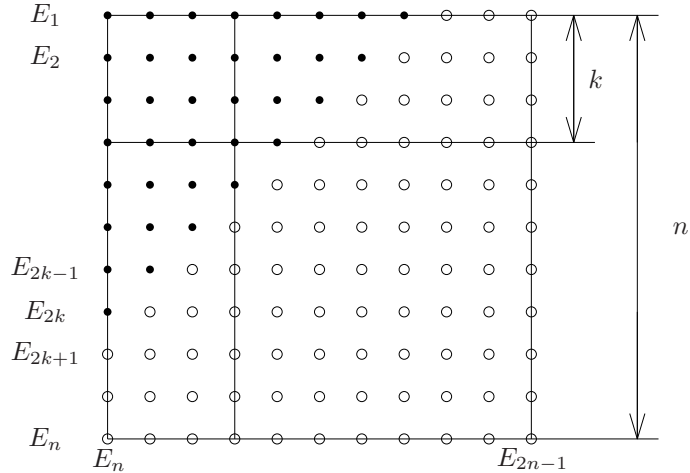}
\end{center}
\caption{Rank illustration. The filled dots represents the Hankel
basis elements which are included in
$\Ran{\partial\mathfrak{W}}(U_B,V_B)$.}
\end{figure}
The reason for working with this simple example, is that it is easy
to generalize the idea to arbitrary $k,n,$ but hard to write down
and we want to omit the details. Roughly, in the general case the
"$V$-derivatives" will span the first $k$ columns of $\m_{n,n}$,
whereas the $U$-derivatives will span the first $k$ rows. Thus
$\Rank{\partial\mathfrak{W}}(U_B,V_B)=2kn-k^2,$ as required in
(\ref{g}). Moreover, it is easy to see that $\{E_1,\ldots,E_{2k}\}$
is a subset of $\Ran{\partial\mathfrak{W}}(U_B,V_B)$, whereas
$\{E_{2k+1},\ldots,E_{2n-1}\}$ form a basis for a disjoint subspace,
(except for the point zero), see Fig \ref{fig7}. In general we thus
get
$$\Rank [{\partial\mathfrak{W}}(U_B,V_B),~~ T_\H]=2kn-k^2+2n-1-2k,$$
as desired.
\end{proof}

\bibliographystyle{plain}
\bibliography{freqest1d_altproj_bibliography}

\begin{thebibliography}{10}

\bibitem{AAK}
V.~M. Adamjan, D.~Z. Arov, and M.~G. Kre{\u\i}n.
\newblock Infinite {H}ankel matrices and generalized problems of
  {C}arath\'eodory-{F}ej\'er and {F}. {R}iesz.
\newblock {\em Funkcional. Anal. i Prilo\v zen.}, 2(1):1--19, 1968.

\bibitem{altprojection}
Fredrik Andersson and Marcus Carlsson.
\newblock Alternating projections of low-dimensional manifolds.
\newblock {\em Submitted}.

\bibitem{JAT}
Fredrik Andersson, Marcus Carlsson, and Maarten~V de~Hoop.
\newblock Sparse approximation of functions using sums of exponentials and aak
  theory.
\newblock {\em Journal of Approximation Theory}, 163(2):213--248, February
  2011.

\bibitem{Bauschke93}
H.~H. Bauschke and J.~M. Borwein.
\newblock On the convergence of von neumann's alternating projection algorithm
  for two sets.
\newblock {\em Set-Valued Analysis}, 1:185--212, 1993.
\newblock 10.1007/BF01027691.

\bibitem{Bauschke96}
Heinz~H. Bauschke and Jonathan~M. Borwein.
\newblock On projection algorithms for solving convex feasibility problems.
\newblock {\em SIAM Rev.}, 38:367--426, September 1996.

\bibitem{Beylkin_Monzon_2005}
Gregory Beylkin and Lucas Monzón.
\newblock On approximation of functions by exponential sums.
\newblock {\em Applied and Computational Harmonic Analysis}, 19(1):17--48, July
  2005.

\bibitem{Bienvenu}
G.~Bienvenu.
\newblock Influence of the spatial coherence of the background noise on high
  resolution passive methods.
\newblock In {\em Acoustics, Speech, and Signal Processing, IEEE International
  Conference on ICASSP '79.}, volume~4, pages 306 -- 309, April 1979.

\bibitem{cadzow}
J.A. Cadzow.
\newblock Signal enhancement-a composite property mapping algorithm.
\newblock {\em Acoustics, Speech and Signal Processing, IEEE Transactions on},
  36(1):49 --62, jan 1988.

\bibitem{structured_low_rank}
Moody~T. Chu, Robert~E. Funderlic, and Robert~J. Plemmons.
\newblock Structured low rank approximation.
\newblock {\em LINEAR ALGEBRA APPL}, 366:157--172, 2002.

\bibitem{eckart_young}
Carl Eckart and Gale Young.
\newblock The approximation of one matrix by another of lower rank.
\newblock {\em Psychometrika}, 1(3):211--218, September 1936.

\bibitem{Golub}
Gene~H. Golub and Charles~F. Van~Loan.
\newblock {\em Matrix computations}.
\newblock Johns Hopkins Studies in the Mathematical Sciences. Johns Hopkins
  University Press, Baltimore, MD, third edition, 1996.

\bibitem{Horn}
Roger~A. Horn and Charles~R. Johnson.
\newblock {\em Topics in matrix analysis}.
\newblock Cambridge University Press, Cambridge, 1994.
\newblock Corrected reprint of the 1991 original.

\bibitem{lewis_malick}
Adrian~S. Lewis and J\'{e}r\^{o}me Malick.
\newblock Alternating projections on manifolds.
\newblock {\em Math. Oper. Res.}, 33:216--234, February 2008.

\bibitem{Liu_altproj_hankel}
Ye~Li, K.J.R. Liu, and J.~Razavilar.
\newblock A parameter estimation scheme for damped sinusoidal signals based on
  low-rank hankel approximation.
\newblock {\em Signal Processing, IEEE Transactions on}, 45(2):481 --486, feb
  1997.

\bibitem{Lu_altproj_hankel}
Biao Lu, Dong Wei, B.L. Evans, and A.C. Bovik.
\newblock Improved matrix pencil methods.
\newblock In {\em Signals, Systems Computers, 1998. Conference Record of the
  Thirty-Second Asilomar Conference on}, volume~2, pages 1433 --1437 vol.2, nov
  1998.

\bibitem{Luk98afast}
Franklin~T. Luk and Sanzheng Qiao.
\newblock A fast eigenvalue algorithm for hankel matrices.
\newblock {\em Linear Algebra Appl}, 316:171--182, 1998.

\bibitem{structured_low_rank_survey}
Ivan Markovsky.
\newblock Structured low-rank approximation and its applications.
\newblock {\em Automatica}, 44:891--909, 2008.

\bibitem{vonNeumann}
John~Von Neumann.
\newblock {\em Functional Operators, Volume II: The Geometry of Orthogonal
  Spaces}.
\newblock Princeton University Press, 1950.

\bibitem{Lanczos_selective}
B.~N. Parlett and D.~S. Scott.
\newblock The {L}anczos algorithm with selective orthogonalization.
\newblock {\em Math. Comp.}, 33(145):217--238, 1979.

\bibitem{Parlett_book}
Beresford~N. Parlett.
\newblock {\em The symmetric eigenvalue problem}.
\newblock Prentice-Hall Inc., Englewood Cliffs, N.J., 1980.
\newblock Prentice-Hall Series in Computational Mathematics.

\bibitem{Pisarenko}
V.~F. Pisarenko.
\newblock The retrieval of harmonics from a covariance function.
\newblock {\em Geophysical Journal of the Royal Astronomical Society},
  33(3):347--366, 1973.

\bibitem{Prabhu_altproj_hankel}
V.U. Prabhu and D.~Jalihal.
\newblock An improved esprit based time-of-arrival estimation algorithm for
  vehicular ofdm systems.
\newblock In {\em Vehicular Technology Conference, 2009. VTC Spring 2009. IEEE
  69th}, pages 1 --4, april 2009.

\bibitem{ESPRIT}
R.~Roy and T.~Kailath.
\newblock {ESPRIT-estimation of signal parameters via rotational invariance
  techniques}.
\newblock {\em IEEE Transactions on Acoustics, Speech, and Signal Processing},
  37(7):984--995, 1989.

\bibitem{Schmidt}
R.~Schmidt.
\newblock Multiple emitter location and signal parameter estimation.
\newblock {\em Antennas and Propagation, IEEE Transactions on}, 34(3):276 --
  280, March 1986.

\bibitem{Lanczos_partial}
Horst~D. Simon.
\newblock The {L}anczos algorithm with partial reorthogonalization.
\newblock {\em Math. Comp.}, 42(165):115--142, 1984.

\bibitem{simoncini_sjostrom}
V.~Simoncini and E.~Sjöström.
\newblock An algorithm for approximating the singular triplets of complex
  symmetric matrices.
\newblock {\em Numerical Linear Algebra with Applications}, 4(6):469--489,
  1997.

\bibitem{Stoica_moses}
Petre Stoica and Randolph Moses.
\newblock {\em Introduction to spectral analysis}.
\newblock Prentice--Hall, 1997.

\bibitem{szego}
G.~Szegö.
\newblock {\em Orthogonal Polynomials}.
\newblock AMS, Providence, RI, 1975.

\bibitem{Fast_SSVD}
Wei Xu and Sanzheng Qiao.
\newblock A fast symmetric {SVD} algorithm for square {H}ankel matrices.
\newblock {\em Linear Algebra Appl.}, 428(2-3):550--563, 2008.

\bibitem{twisted_takagi}
Wei Xu and Sanzheng Qiao.
\newblock A twisted factorization method for symmetric {SVD} of a complex
  symmetric tridiagonal matrix.
\newblock {\em Numer. Linear Algebra Appl.}, 16(10):801--815, 2009.

\bibitem{Zangwill}
W.~I. Zangwill.
\newblock {\em Nonlinear Programming}.
\newblock Prentice Hall, Englewood Cliffs, N. J., 1969.

\end{thebibliography}

\end{document}